\documentclass[reqno]{amsart}

\usepackage{amssymb,amsmath,amsthm}
\usepackage{color}
\usepackage{enumerate}
\usepackage{fullpage}
\usepackage{url}
\usepackage[Symbol]{upgreek}
\usepackage{hyperref}
\usepackage[T1]{fontenc}
\usepackage{mathtools}
\usepackage{tikz}
\usetikzlibrary{cd}
\usepackage{xspace} 
\usepackage{comment}

\numberwithin{equation}{section}
\newtheorem{lemma}{Lemma}[section]
\newtheorem{thm}[lemma]{Theorem}
\newtheorem{corollary}[lemma]{Corollary}

\newtheorem*{conjecture*}{Conjecture}
\newtheorem{proposition}[lemma]{Proposition}
\newtheorem{open problem}[lemma]{Open problem}

\newtheorem*{fact*}{Fact}
\newtheorem{fact}[lemma]{Fact}

\newtheorem*{claim*}{Claim}

\newtheorem{thmA}{Theorem}

\newcommand{\thistheoremname}{}
\newtheorem*{genericthm*}{\thistheoremname}
\newenvironment{namedthm*}[1]
{\renewcommand{\thistheoremname}{#1}%
\begin{genericthm*}}
{\end{genericthm*}}

\theoremstyle{definition}

\newtheorem{definition}[lemma]{Definition}

\newtheorem{remark}[lemma]{Remark}
\newtheorem{final remark}[lemma]{Final remark}
\newtheorem{example}[lemma]{Example}

\DeclareMathOperator{\acl}{acl}
\DeclareMathOperator{\Aff}{Aff}
\DeclareMathOperator{\cl}{c\ell}

\DeclareMathOperator{\dcl}{dcl}

\DeclareMathOperator{\eq}{eq}

\DeclareMathOperator{\Jet}{Jet}

\DeclareMathOperator{\ring}{ring}
\DeclareMathOperator{\rk}{rk}

\DeclareMathOperator{\trdeg}{trdeg}
\DeclareMathOperator{\Th}{Th}
\DeclareMathOperator{\tp}{tp}
\DeclareMathOperator{\Inv}{Inv}

\newcommand{\M}{\mathbb{M}}

\newcommand{\N}{\mathbb{N}}
\newcommand{\PP}{\mathbb{P}}

\newcommand{\R}{\mathbb{R}}

\newcommand{\cC}{\mathcal C}

\newcommand{\cL}{\mathcal L}

\newcommand{\cP}{\mathcal P}

\newcommand{\cX}{\mathcal X}
\newcommand{\cY}{\mathcal Y}
\newcommand{\cZ}{\mathcal Z}

\newcommand{\av}{{\bar a}}
\newcommand{\bv}{{\bar b}}
\newcommand{\cv}{{\bar c}}
\newcommand{\dv}{{\bar d}}

\newcommand{\rv}{{\bar r}}

\newcommand{\uv}{{\bar u}}

\newcommand{\x}{{\bar x}}
\newcommand{\y}{{\bar y}}
\newcommand{\z}{{\bar z}}

\newcommand{\zero}{{\bar 0}}

\newcommand{\aclL}{\acl_{\cL}}
\newcommand{\acleq}{\acl_{\cL}^{\eq}}
\newcommand{\aclD}{\acl_{\Delta}}
\newcommand{\clD}{\cl^\Delta}

\newcommand{\dclL}{\dcl_{\cL}}
\newcommand{\dclD}{\dcl_{\Delta}}

\newcommand{\dimL}{\dim_{\cL}}

\newcommand{\On}{\mathbf{On}}
\newcommand{\LD}{\cL^\Delta}

\newcommand{\rkD}{\rk^\Delta}
\newcommand{\rkL}{\rk_{\cL}}
\newcommand{\rkeq}{\rk_{\cL}^{\eq}}

\newcommand{\TD}{T^{\Delta}}
\newcommand{\TDn}{T^{\Delta,\operatorname{nc}}}
\newcommand{\Tb}{T^{\Delta,?}}

\newcommand{\TDG}{\TD_g}
\newcommand{\TDnG}{\TDn_g}
\newcommand{\TG}{\Tb_g}
\newcommand{\tpL}{\tp_{\cL}}
\newcommand{\tpD}{\tp_{\LD}}

\newcommand{\Uth}{U^{\th}}

\DeclareTextSymbol{\thh}{T1}{254}
\def\th{\textnormal{\thh}}

\makeatletter
\def\indsym#1#2{%
  \setbox0=\hbox{$\m@th#1x$}%
  \kern\wd0%
  \hbox to 0pt{\hss$\m@th#1\mid$\hbox to 0pt{$\m@th#1^{#2}$}\hss}%
  \lower.9\ht0\hbox to 0pt{\hss$\m@th#1\smile$\hss}%
  \kern\wd0}
\newcommand{\ind}[1][]{\mathop{\mathpalette\indsym{#1}}}

\def\nindsym#1#2{%
  \setbox0=\hbox{$\m@th#1x$}%
  \kern\wd0%
  \hbox to 0pt{\hss$\m@th#1\not$\kern1.4\wd0\hss}
  \hbox to 0pt{\hss$\m@th#1\mid$\hbox to 0pt{$\m@th#1^{\,#2}$}\hss}%
  \lower.9\ht0\hbox to 0pt{\hss$\m@th#1\smile$\hss}%
  \kern\wd0}
\newcommand{\nind}[1][]{\mathop{\mathpalette\nindsym{#1}}}

\def\thind{\ind[\th]}
\def\Tind{\ind[\operatorname{M}]} 
\def\eqind{\ind[\eq]}
\def\find{\ind[\operatorname{f}]}

\def\Dind{\ind[\Delta]}
\renewcommand{\preceq}{\preccurlyeq}

\renewcommand{\leq}{\leqslant}
\renewcommand{\geq}{\geqslant}
\renewcommand{\epsilon}{\varepsilon}

\author{Antongiulio Fornasiero}
\address{Dipartimento di Matematica e Informatica ``Ulisse Dini,'' Viale Morgagni, 67/a, 50134 Firenze, Italy}
\email{antongiulio.fornasiero@gmail.com}

\author{Elliot Kaplan}
\address{Universit\'{e} de Mons, D\'{e}partement de Math\'{e}matique, Avenue Maistriau 15, 7000 Mons, Belgium}
\email{elliot.kaplan@umons.ac.be}

\author{Angus Matthews}
\address{School of Mathematics, University of Leeds, Leeds LS2 9JT, United Kingdom}
\email{mmamat@leeds.ac.uk}

\title{Geometric fields, ranks, and generic derivations}

\date{\today}


\begin{document}
\maketitle
\begin{abstract}
In this note, we show various minimality results for a geometric theory of fields $T$: $T$ is stable if and only if it is strongly minimal, $T$ is simple if and only if it has SU-rank~1, and $T$ is rosy if and only if $T$ is surgical.
Combining the first equivalence with an earlier result of Hrushovski, we deduce that algebraically bounded stable fields are precisely expansions of algebraically closed fields by constants.

We then consider algebraically bounded and o-minimal expansions of fields with generic derivations.
We show that if $\M$ is a simple algebraically bounded structure and $\Delta$ is a generic tuple of derivations on $\M$, then $(\M;\Delta)$ is supersimple if and only if the derivations commute.
Similarly, if $\M$ is an o-minimal structure and $\Delta$ is a generic tuple of $T$-derivations on $\M$, then $(\M;\Delta)$ is superrosy if and only if the derivations commute.
We obtain explicit bounds on ranks using the Kolchin polynomial.
\end{abstract}

\section*{Introduction}
\subsection*{Geometric theories of fields}
In many interesting fields and expansions of fields, the model-theoretic algebraic closure $\acl$ has the exchange property: $a \in \acl(A,b)\setminus \acl(A) \Rightarrow b \in \acl(A,a)$.
Examples include:
\begin{enumerate}
\item All algebraically bounded fields (for example, algebraically closed and pseudofinite fields)~\cite{vdD89}.
\item All o-minimal expansions of ordered fields~\cite[Theorem 4.1]{PS86}.
\item All $C$-minimal expansions of valued fields~\cite{Jo24}.
\item All 1-h-minimal expansions of characteristic zero valued fields.
In the equicharacteristic case, this is~\cite[Lemma 5.3.5]{CHRK22}.
The mixed characteristic case follows by considering the finest equicharacteristic coarsening and applying~\cite[Lemma 2.5.3]{CHRKV23}.
\end{enumerate}
These fields necessarily eliminate the quantifier $\exists^\infty$ (any definable family of finite sets has a bound on the cardinalities of the sets), and so definable sets come equipped with a definable dimension.
We call a theory of fields in which $\acl$ always satisfies exchange a \textbf{geometric theory of fields}.

Often, exchange for $\acl$ follows from some form of minimality.
For example, if $T$ is strongly minimal then $\acl$ satisfies exchange in all models of $T$.
More generally, if $T$ is supersimple of SU-rank~1 (or even superrosy of \thh-rank~1), then $\acl$ satisfies exchange.
It is folklore that superrosy theories of \thh-rank~1 are precisely the \emph{surgical theories} of Gagelman~\cite{Ga05}; we provide a proof in Appendix~\ref{sec:appendixa}.
Of course, many geometric theories of fields are not rosy; any field with a nontrivial definable valuation is not rosy~\cite[Fact 1.8]{Kr15}.

In the first part of this note, we show that for theories of fields, being geometric is actually \emph{equivalent} to minimality under appropriate assumptions.
More precisely, we show:

\begin{thmA}[Theorems~\ref{thm:stablestronglyminimal},~\ref{thm:simple_supersimple_rank_1}, and~\ref{thm:rosy_surgical}]\label{theoremA}
Let $T$ be a geometric theory of fields.
\begin{enumerate}
\item If $T$ is stable, then $T$ is strongly minimal.
\item If $T$ is simple, then $T$ is supersimple of SU-rank~1.
\item If $T$ is rosy, then $T$ is surgical.
\end{enumerate}
\end{thmA}

We note that these minimality results do not hold for Kim-independence in NSOP$_1$-theories; see Remark~\ref{rem:kim}.

A theory $T$ is said to be \textbf{algebraically bounded} if model-theoretic algebraic closure agrees with field-theoretic algebraic closure over $\PP\coloneqq \dcl(\emptyset)$ in every model $\M\models T$.
That is, if 
\[
a \in \aclL(B) \iff \trdeg(a|\PP(B)) = 0
\]
for $a\in \M$ and $B \subseteq \M$.
This is not van den Dries' original definition of algebraic boundedness, but it is equivalent~\cite{JY23}.
Combining Theorem~\ref{theoremA} with an earlier result of Hrushovski~\cite{Hr92}, we obtain:

\begin{namedthm*}{Corollary~\ref{cor:stable_acf_constants}}
Any algebraically bounded stable field is an expansion of an algebraically closed field by constants.
\end{namedthm*}

On the other hand, Hrushovski demonstrates that there are many strongly minimal expansions of an algebraically closed field (all of which are necessarily geometric)~\cite{Hr92}.
The rigidity observed above for algebraically bounded stable fields does not hold for simple fields.
Indeed, any bounded perfect pseudo-algebraically closed (PAC) field is algebraically bounded and supersimple of SU-rank~1~\cite{Hr02}, as is the expansion of an algebraically closed field by a generic predicate~\cite{CP98}.
Of course, it is conjectured that the only supersimple pure fields are bounded perfect PAC which, combined with Theorem~\ref{theoremA}, implies that these are also the only geometric simple pure fields.
There is a similar conjecture for superrosy pure fields, where PAC is replaced by \emph{pseudo-real closed}~\cite{Kr15}.
This can be similarly extended to geometric rosy pure fields using Theorem~\ref{theoremA}.

\subsection*{Expansions by generic derivations}
Beginning in Section~\ref{sec:derivations}, we turn our attention to ranks on geometric fields (always of characteristic zero) equipped with a generic tuple of derivations.
As of now, there is no general framework for expanding a geometric field by derivations and obtaining a model completion (the naive approach fails in general, see~\cite{FT25b}), though there are some frameworks in which things can be done \emph{assuming} the existence of a model companion~\cite{LSM25,Po25}.
While our methods can likely be generalized to these frameworks, we restrict our attention to the two settings where the expansion by a tuple of derivations is known to admit a model completion: algebraically bounded structures~\cite{FT25} and o-minimal structures~\cite{FK21}.
Note that in~\cite{FK21}, the model companion is only shown to exist for o-minimal structures with several \emph{commuting} compatible derivations; though, as we show in Appendix~\ref{sec:appendix}, the model companion also exists in the noncommuting case.

Let $T$ be an algebraically bounded or o-minimal theory extending the theory of fields of characteristic zero.
Let $\TDn$ be the theory of models of $T$ equipped with a finite tuple $\Delta$ of derivations (compatible in the o-minimal case; see Definition~\ref{def:Tderivation}), and let $\TD$ extend $\TDn$ by axioms stating that the derivations in $\Delta$ commute.
Both $\TD$ and $\TDn$ have model completions, which we denote by $\TDG$ and $\TDnG$ respectively.
In Corollary~\ref{cor:derivation-like}, we show that the theories $\TD$ and $\TDn$ fit into the framework of \emph{derivation-like theories} from~\cite{LSM25}.
We use this to deduce the following:

\begin{thmA}[Corollaries~\ref{cor:Dindcoincide},~\ref{cor:simple}, and~\ref{cor:NSOP1}]\label{theoremB}\
\begin{enumerate}
\item\label{B1} If $T$ is simple, then so are $\TDG$ and $\TDnG$.
\item\label{B2} If $T$ is NSOP$_1$, then so are $\TDG$ and $\TDnG$.
\item\label{B3} If $\TDG$ has GEI, then it is rosy; likewise for $\TDnG$.
\end{enumerate}
\end{thmA}

Note that~\eqref{B1} and~\eqref{B2} hold vacuously for o-minimal theories, and that~\eqref{B1} and~\eqref{B3} were already established for algebraically bounded theories in~\cite{FT25c}.
In the o-minimal case,~\eqref{B3} holds unconditionally, since $\TDG$ and $\TDnG$ are known to eliminate imaginaries; see~\cite[Theorem 5.19]{FK21} and Corollary~\ref{cor:ominEI}.
It is conjectured that GEI always transfers from $T$ to $\TDG$ and $\TDnG$~\cite[Conjecture 3.7]{FT25c}, a result that is known to hold for certain topological theories~\cite{CKP23,FT25c}.
The results in Theorem~\ref{theoremB} contribute to a growing list of model-theoretic properties known to transfer from $T$ to $\TDG$ and $\TDnG$, including stability, NIP, and distality~\cite{FT25,FT25c,FK21}, as well as NTP$_2$ and the antichain tree property in the case of a single derivation~\cite{KK26}.

\subsection*{Generic derivations and ranks} 
In Section~\ref{sec:ranks}, we turn to an analysis of ranks in $\TDG$ and $\TDnG$.
Let $m \coloneqq |\Delta|$ denote the number of derivations.
First, we investigate supersimplicity, superstability, and superrosiness, showing that these transfer in the commuting case:

\begin{thmA}[Proposition~\ref{prop:notranked}, Theorem~\ref{thm:Urank}]\label{theoremC}
If $m\geq 2$, then $\TDnG$ is never superrosy (thus, never supersimple or superstable).
On the other hand:
\begin{enumerate}
\item\label{C1} If $T$ is simple, then $\TDG$ is supersimple of SU-rank~$\omega^m$.
\item\label{C2} If $\TDG$ has GEI, then $\TDG$ is superrosy of \thh-rank~$\omega^m$.
\end{enumerate}
\end{thmA}

The key tool in proving Theorem~\ref{theoremC} is the Kolchin polynomial (or an appropriate modification thereof in the o-minimal case~\cite{FK24}).
It has long been known that the Kolchin polynomial can be used to bound ranks in differentially closed fields~\cite{McG00}, and we show that SU-rank and \thh-rank can be bounded by the Kolchin polynomial in our setting as well.

Following this, we characterize when $\TDG$ and $\TDnG$ are strongly dependent or, more broadly, strong.
Strong theories are a subclass of NTP$_2$ theories, and strongly dependent theories are those strong theories that are also NIP.

\begin{thmA}[Proposition~\ref{prop:notstrong}, Theorem~\ref{thm:strong}]\label{theoremD}
If $m>1$, then $\TDnG$ is never strong.
The theory $\TDG$ is strong (respectively, strongly dependent) if and only if $T$ is simple (respectively, stable).
\end{thmA}

Note that in the case that $T$ is stable, then Corollary~\ref{cor:stable_acf_constants} tells us that it must be an expansion of ACF$_0$ by constants.
Thus, $\TDG$ and $\TDnG$ are expansions of DCF$_{0,m}$ and DCF$^{\operatorname{nc}}_{0,m}$ by constants, respectively.
It follows from known results about differentially closed fields that $\TDG$ is superstable of U-rank~$\omega^m$, and thus strong~\cite{McG00}.

A natural question is whether the GEI assumptions in Theorems~\ref{theoremB}\eqref{B3} and~\ref{theoremC}\eqref{C2} can be dropped.
As mentioned, this would follow from~\cite[Conjecture 3.7]{FT25c}, but it would also follow from the weaker conjecture:

\begin{conjecture*}
If $T$ is rosy and algebraically bounded (thus, surgical), then $\TDG$ and $\TDnG$ have GEI.
\end{conjecture*}

\subsection*{Acknowledgements}
The authors would like to thank Vincenzo Mantova, James Freitag, and Giuseppina Terzo for helpful conversations.
The authors would also like to thank Erik Walsberg for helping them find a faster proof of Theorem~\ref{thm:stablestronglyminimal}.
The authors acknowledge financial support from INdAM (Istituto Nazionale di Alta Matematica), particularly through the Intensive Period in Model Theory.

Elliot Kaplan is supported by a CR fellowship through the Fonds de la Recherche Scientifique -- FNRS, and was previously supported by the National Science Foundation under Award No.\ DMS-2103240.
He thanks the Fields Institute and the Max Planck Institute for Mathematics for their support and hospitality while conducting this research.
Angus Matthews's work is supported by EPSRC DTP 2224 [EP/W524372/1] (at the University of Leeds).

\section{Independence relations}
Let $T$ be a complete $\cL$-theory with infinite models, and let $\M \models T$ be a monster model.
We will use $A$, $B$, and $C$ to denote small subsets of $\M$.

\begin{definition}[Adler~\cite{Ad09}]
An \textbf{independence relation} on $\M$ is a ternary relation $\ind$ between small subsets of $\M$ which satisfies the following axioms:
\begin{enumerate}[(1)]
\item \textbf{Invariance:} If $A \ind_C B$ and $(A^*,B^*,C^*) \equiv_{\cL}(A,B,C)$, then $A^* \ind_{C^*} B^*$.
\item \textbf{Monotonicity:} If $A \ind_C B$, $A_0 \subseteq A$, and $B_0 \subseteq B$, then $A_0\ind_C B_0$.
\item \textbf{Base monotonicity:} Suppose $C_0 \subseteq C \subseteq B$.
If $A \ind_{C_0}B$, then $A \ind_C B$.
\item \textbf{Transitivity:} Suppose $D \subseteq C \subseteq B$.
If $B \ind_CA$ and $C\ind_DA$, then $B\ind_DA$.
\item \textbf{Normality:} If $A \ind_C B$, then $AC\ind_C B$.
\item \textbf{Finite character:} If $A_0 \ind_C B$ for all finite $A_0 \subseteq A$, then $A \ind_C B$.
\item \textbf{Local character:} For each $A$, there is a cardinal $\kappa(A)$ such that for all $B$, there is $C \subseteq B$ of cardinality $|C|<\kappa(A)$ with $A \ind_C B$.
\item \textbf{Full existence:} For all $A,B,C$, there is $A^* \equiv_{\cL(C)}A$ with $A^* \ind_C B$.
\item \textbf{Symmetry:} If $A \ind_C B$, then $B \ind_C A$.
\end{enumerate}
An independence relation $\ind$ is \textbf{real-strict} if it satisfies:
\begin{enumerate}[(10)]
\item \textbf{Anti-reflexivity:} For all $a \in \M$, if $a \ind_B a$, then $a \in \aclL(B)$.
\end{enumerate}
A \textbf{real-canonical independence relation} is a real-strict independence relation which also satisfies:
\begin{enumerate}[(11)]
\item \textbf{Intersection:} for $C_1,C_2 \subseteq B$ and $D \coloneqq\aclL(C_1) \cap \aclL(C_2)$, we have
\[\textstyle
A\ind_{C_1}B\text{ and }A \ind_{C_2}B \implies A \ind_D B.
\]
\end{enumerate}
An independence relation is said to be \textbf{strict} if it extends to an independence relation on $\M^{\eq}$ satisfying anti-reflexivity (where $\aclL$ is replaced with $\aclL^{\eq}$).
Likewise, a strict independence relation is \textbf{canonical} if its extension to $\M^{\eq}$ satisfies intersection.
\end{definition}

\begin{remark}
In Adler's definition, full existence and symmetry are replaced by
\begin{enumerate}[]
\item \textbf{Extension:} If $A \ind_C B$ and $B^* \supseteq B$, then there is $A^* \equiv_{\cL(BC)} A$ with $A^* \ind_CB^*$.
\end{enumerate}
By~\cite[Remark 1.2 and Theorem 2.5]{Ad09}, full existence and symmetry are equivalent to extension, modulo the other axioms.
By~\cite[Proposition 3.1.3]{dE23}, the axioms also imply
\begin{enumerate}[]
\item \textbf{Strong closure:} $A \ind_CB \iff \aclL(AC)\ind_{\aclL(C)} \aclL(BC)$.
\end{enumerate}
\end{remark}

\begin{remark}\label{rem:GEIstrict}
If $T$ has geometric elimination of imaginaries (GEI), then $\ind$ is real-strict if and only if it is strict, and $\ind$ is real-canonical if and only if it is canonical.
\end{remark}

Given independence relations $\ind[1]$ and $\ind[2]$ on $\M$, we say that \textbf{$\ind[2]$ is coarser than $\ind[1]$} if
\[\textstyle
A \ind[1]_CB \implies A\ind[2]_CB
\]
for all $A,B,C$.

Following Adler~\cite{Ad09}, we say that $T$ is \textbf{rosy} if $\M$ admits a strict independence relation.
By~\cite[Proposition 4.7]{Ad09}, this agrees with Onshuus' original definition of rosiness; see~\cite{On06, EO07}.

\begin{fact}[{\cite[Theorem 4.3]{Ad09},~\cite[Theorem 3.3]{Ad05}}]\label{fact:coincide}
If $\M^{\eq}$ admits a strict independence relation $\ind$, then it admits a coarsest strict independence relation $\thind$ called \textbf{\thh-forking independence}.
If $\ind$ is canonical, then it coincides with $\thind$.
\end{fact}

\begin{remark}
Adler~\cite{Ad09} defines a theory to be \textbf{real rosy} if $\M$ admits a real-strict independence relation.
If $T$ is real rosy, then $\M$ admits a coarsest real-strict independence relation, called \textbf{real \thh-forking independence}.
Note that this need not coincide with the restriction of \thh-forking independence to real sets, even in rosy theories.
As we will see in Lemma~\ref{lem:Tind-coarsest} and Proposition~\ref{prop:Dind-coarsest} below, we can describe real \thh-forking independence quite explicitly in the theories we consider.
\end{remark}

\begin{remark}\label{rem:canonicaliscoarsest}
Following the proof of~\cite[Lemma 3.2]{Ad05}, we see that if $\ind$ is a real-canonical independence relation on $\M$, then it is the coarsest real-strict independence relation on $\M$, and so coincides with real \thh-forking.
\end{remark}

Let $\find$ denote forking independence in $T$.

\begin{fact}[{\cite[Theorem 5.3]{Ad09}}]
$T$ is simple if and only if $\find$ is an independence relation on $\M$.
If $T$ is simple, then any independence relation on $\M$ is coarser than $\find$.
\end{fact}

\subsection{Ranks}
Let $\ind$ be an independence relation on $\M$ and let $p$ be an
$\cL(A)$-type.
An $\cL(B)$-type $q$ is an \textbf{$\ind$-forking extension of $p$} if $B \supseteq A$, $q$ extends $p$, and $\av\nind_AB$ for some (equivalently,
any) tuple $\av$ realizing $q$.
The \textbf{$\ind$-rank of $p$}, denoted $U^{\tiny\! \ind\!}(p)$, is defined as follows:
\begin{enumerate}
\item We always have $U^{\tiny\! \ind\!}(p)\geq 0$, as $p$ is assumed to be consistent;
\item For $\alpha$ any ordinal, $U^{\tiny\! \ind\!}(p) \geq \alpha+1$ if there is an $\ind$-forking extension $q$ of $p$ with $U^{\tiny\! \ind\!}(q) \geq \alpha$;
\item For $\lambda$ any nonzero limit ordinal, $U^{\tiny\! \ind\!}(p) \geq \lambda$ if $U^{\tiny\! \ind\!}(p)\geq \alpha$ for all ordinals $\alpha< \lambda$.
\end{enumerate}
If there is an ordinal $\alpha$ with $U^{\tiny\! \ind\!}(p) \geq \alpha$ and $U^{\tiny\! \ind\!}(p) \not\geq \alpha+1$, then we put $U^{\tiny\! \ind\!}(p) \coloneqq \alpha$.
Otherwise, we put $U^{\tiny\! \ind\!}(p)\coloneqq \infty$.
We say that $\ind$ is \textbf{ranked} if $U^{\tiny\! \ind\!}(p) <\infty$ for all types $p$.
Note that if $\ind$ is real-strict and $U^{\tiny\! \ind\!}(p) = 0$, then $p$ is algebraic.

For $\av \in \M^n$, we put $U^{\tiny\! \ind\!}(\av|A) \coloneqq U^{\tiny\! \ind\!}(\tpL(\av|A))$.
For an $\cL(A)$-definable set $X$, we set
\[
U^{\tiny\! \ind\!}(X) \coloneqq \sup\{U^{\tiny\! \ind\!}(\av|A):\av \in X\}.
\]
We also set
\[
U^{\tiny\! \ind\!}(T) \coloneqq \sup\{U^{\tiny\! \ind\!}(p):p\text{ is a unary $\cL(\emptyset)$-type}\}.
\]
Following the proof of~\cite[Theorem 5.1.6]{Wa00}, one can establish the Lascar inequalities for the $\ind$-rank (no changes to the proof are needed).
\begin{fact}[Lascar inequalities]
For $a,b \in \M$, we have
\[
U^{\tiny\! \ind\!}(a|A)+ U^{\tiny\! \ind\!}(b|A\cup\{a\}) \leq U^{\tiny\! \ind\!}(a,b|A)\leq U^{\tiny\! \ind\!}(a|A)\oplus U^{\tiny\! \ind\!}(b|A\cup\{a\}),
\]
where $+$ is the ordinal sum and $\oplus$ is the commutative (Hessenberg) sum, it being understood that both sums take value $\infty$ if some summand is $\infty$.
\end{fact}

\begin{corollary}
The following are equivalent:
\begin{enumerate}
\item $\ind$ is ranked;
\item $U^{\tiny\! \ind\!}(p)<\infty$ for all unary types $p$;
\item $U^{\tiny\! \ind\!}(T)<\infty$.
\end{enumerate}
\end{corollary}

\begin{fact}\label{fact:Uorder}
Suppose $\ind[1]$ and $\ind[2]$ are independence relations on $\M$ and that $\ind[2]$ is coarser than $\ind[1]$.
Then for any type $p$, we have
\[
U^{\tiny\! \ind[1]}(p)\geq U^{\tiny\! \ind[2]}(p).
\]
In particular, if $\ind[1]$ is ranked, then so is $\ind[2]$.
\end{fact}

In many cases of interest, we have specific names for the $\ind$-rank:
\begin{enumerate}
\item If $T$ is simple, then $\find$-rank is called \textbf{SU-rank}, and the SU-rank of $p$ is denoted $SU(p)$.
If in addition $\find$ is ranked, then $T$ is said to be \textbf{supersimple}.
\item If $T$ is stable, then SU-rank agrees with Lascar's U-rank.
If $T$ is also supersimple, then it is said to be \textbf{superstable}.
\item If $\thind$ is an independence relation, then $\thind$-rank is called \textbf{\thh-rank}, and the \thh-rank of $p$ is denoted $\Uth(p)$.
If $T$ is rosy and $\thind$ on $\M^{\eq}$ is ranked, then $T$ is said to be \textbf{superrosy}.
\end{enumerate}

\subsection{Independence in pregeometric theories}
\label{sec:ind-pregeom}
Recall that $T$ is \textbf{pregeometric} if $\aclL$ is a pregeometry in any model of $T$, and that $T$ is \textbf{geometric} if it is pregeometric and eliminates $\exists^\infty$.
Suppose that $T$ is pregeometric and let $\rkL$ denote the rank function corresponding to $\aclL$.
We define the \textbf{$\aclL$-independence relation} $\Tind\ $ as follows:
\[\textstyle
A\Tind_C B \iff \rkL(A_0|BC) = \rkL(A_0|C)\text{ for every finite subset }A_0 \subseteq A.
\]

The following fact is well-known; we prove a stronger version in Proposition~\ref{prop:eqindisind} below:

\begin{fact}\label{fact:Tindrealstrict}
Suppose that $T$ is pregeometric.
Then $\Tind\ $ is a real-strict independence relation on $\M$, and the $\Tind\ $-rank of $\tp(\av|B)$ coincides with $\rkL(\av|B)$ for all tuples $\av\in \M^n$.
In particular, $T$ has $\Tind\ $-rank~1.
\end{fact}

\begin{lemma}\label{lem:Tind-coarsest}
Suppose $T$ is pregeometric.
Then $\Tind\ $ is the coarsest real-strict independence relation on $\M$, and so coincides with real \thh-forking independence.
\end{lemma}
\begin{proof}
Let $\ind$ be a real-strict independence relation on $\M$.
By monotonicity and finite character, it suffices to show that $A \ind_CB \implies A\Tind_CB$ for all finite $A$.
Let us first suppose that $a\ind_CB$ for a single element $a \in \M$.
If $\rkL(a|BC) = 1$, then $\rkL(a|C) = 1$ as well.
On the other hand, if $a \in \aclL(BC)$, then strong closure gives $\aclL(aC)\ind_{\aclL(C)}\aclL(BC)$, monotonicity gives $a \ind_{\aclL(C)}a$, and real-strictness gives $a \in \aclL(C)$.
In both cases, we have $a\Tind_CB$.

Now we show that $A \ind_CB \implies A\Tind_CB$ for $A$ finite by induction on the size of $A$.
Suppose this holds for a given $A$, and let $a \in \M$ with $Aa\ind_CB$.
Normality, base monotonicity, and monotonicity give $A \ind_CB$ and $a \ind_{AC}B$ (where we use symmetry to apply base monotonicity on the left).
Our induction hypothesis gives $A \Tind_CB$, and the singleton case gives $a \Tind_{AC}B$.
Applying normality, we have $AC \Tind_{C}B$ and $aAC \Tind_{AC}B$, and applying transitivity yields $aAC\Tind_CB$.
We conclude by monotonicity.
\end{proof}

Gagelman defines a \textbf{surgical theory} to be a pregeometric theory with the additional property that for any definable equivalence relation $E$ on a definable set $X$, only finitely many $E$-classes have the same $\aclL$-dimension as $X$~\cite{Ga05}.
It is folklore that a theory $T$ is surgical if and only if it is superrosy of \thh-rank~1; we provide a proof in Proposition~\ref{prop:surgicalrosy} below.
Accordingly, we use \emph{surgical} to mean \emph{superrosy of \thh-rank~1}.
If $T$ is surgical, then $A\Tind_CB\iff A\thind_CB$ for all $A,B,C\subseteq \M$.
If $T$ is pregeometric with GEI, then $T$ is surgical by~\cite[Corollary 3.6]{Ga05}.
Being pregeometric alone is not enough:

\begin{example}
Let $\M$ be an algebraically closed valued field with a
nontrivial valuation.
Then $T$ is geometric and $\Tind\ $ is real-strict and real-canonical, but $T$ is not rosy.
\end{example}

For simple theories with elimination of hyperimaginaries (for example, stable and supersimple theories), forking independence coincides with \thh-forking independence; see~\cite[Corollary 3.4]{Ad05} or~\cite[Theorem~5.1.4]{On06}.
By Proposition~\ref{prop:surgicalrosy}, this gives:

\begin{fact}\label{fact:supersimplecoincide}
If $T$ is supersimple with $SU(T) = 1$, then $T$ is surgical and $\find = \thind = \Tind\ $.
\end{fact}

The SU-rank~1 assumption is necessary:

\begin{example}
Let $\M = (M;E)$, where $M$ is a set and $E$ is an equivalence relation on $M$ with infinitely many equivalence classes, all infinite.
Then $T$ is $\omega$-stable of Morley rank~2, geometric, and $\Tind\ $ is real-strict and real-canonical, but $\find = \thind\neq \Tind\ $.
\end{example}

\section{Geometric theories of fields}
We assume for the rest of this note that $T$ is a pregeometric theory extending the $\cL_{\ring}$-theory of fields.
Throughout this section, \emph{independence} refers to $\aclL$-independence (equivalently, $\Tind\ $-independence).
Given a definable set $Y\subseteq \M^n$, defined over a small set of parameters $A$, we set 
\[
\dimL(Y)\coloneqq \max\{\rkL(\av|A): \av\in Y\}.
\]
Then $\dimL(Y)$ is independent of the specific choice of defining parameters.

Let $X \subseteq \M$ be an infinite definable set.
For $\alpha \in \M$, we let
\[
\Inv_X(\alpha) \coloneqq \{(a,b,c,d) \in X^4: (a-d)/(b-c) = \alpha\}.
\]
Given $S \subseteq \M$, we let $\Inv_X(S) \coloneqq \bigcup_{\alpha \in S}\Inv_X(\alpha)$.
In~\cite[Lemma 3.47]{Fo11}, it is shown that $\Inv_X(\alpha)$ is nonempty for each $\alpha$; see also~\cite{DMS10,JY23}.
It follows that $T$ eliminates $\exists^\infty$, and so we refer to $T$ as a \textbf{geometric theory of fields}.
Since $T$ is geometric, $\dimL$ is a definable dimension function in the sense of~\cite{vdD89}.
Note that $\dimL \Inv_X(\alpha) \leq 3$; this is even an equality:

\begin{lemma}\label{lem:very_strong_JY}
Let $\alpha \in \M$ and let $(a,b) \in X^2$ be independent over $\alpha$ and the defining parameters for $X$.
Then the fiber
\[
\Inv_X(\alpha)_{a,b} = \{(c,d) \in X^2:(a-d)/(b-c) = \alpha\}
\]
is infinite.
In particular, $\Inv_X(\alpha)$ has dimension 3.
\end{lemma}
\begin{proof}
Fix a set of parameters $A$ such that $X$ is $\cL(A)$-definable and $\rkL(a,b| A,\alpha) = 2$.
We have $(c,d) \in \Inv_X(\alpha)_{a,b}$ if and only if $a-b\alpha = d-c\alpha$ and $(c,d) \neq (b,a)$.
Thus, it suffices to show that the set $\{(c,d) \in X^2: a-b\alpha = d-c\alpha\}$ is infinite.
This holds since $\rkL(a,b|A,\alpha,a-b\alpha) = 1$.
\end{proof}

The following consequence of Lemma~\ref{lem:very_strong_JY} was used in an earlier proof of Theorem~\ref{thm:stablestronglyminimal} below.
Although it is no longer needed for this theorem, we include it here as it may be of independent interest.

\begin{lemma}\label{lem:almost_field_is_field}
Let $X\subseteq \M$ be infinite and definable over $A \subseteq \M$.
Let $\alpha \in \M^\times$ and suppose that for all $\beta,\gamma \in X$ with $\rkL(\beta,\gamma|A,\alpha) = 2$, we have $\beta-\gamma,\alpha\beta/\gamma \in X$.
Then $X$ is cofinite in $\M$.
\end{lemma}
\begin{proof}
Fix some $\delta \in \M$ with $\rkL(\delta|A,\alpha) = 1$.
It suffices to show that $\delta \in X$.
Using Lemma~\ref{lem:very_strong_JY}, take $(a,b,c,d) \in \Inv_X(\alpha^{-1}\delta)$ with $\rkL(a,b,c,d|A,\alpha,\delta) \geq 3$.
Then $\rkL(a,b,c,d|A,\alpha) = \rkL(a,b,c,d,\delta|A,\alpha) = 4$.
Applying our hypothesis gives $a-d,b-c \in X$, and so $\delta = \alpha(a-d)/(b-c) \in X$ as well.
\end{proof}

Now let $\cY = (Y_z)_{z \in Z}$ be a definable family of infinite subsets of $\M$.
We allow the indexing set $Z$ to be a definable set of imaginaries.
For $a,b\in X$ and $z \in Z$, we let
\[
X^\cY_{a,b,z} \coloneqq \{c \in X: X \cap (a+(c-b)Y_z)\text{ is infinite}\}.
\]

\begin{lemma}\label{lem:sufficiently_cover_implies_many_subsets}
Let $(a,b) \in X^2$ be independent over the defining parameters for $X$ and $\cY$.
Then for any $z \in Z$, the set $X^\cY_{a,b,z}$ is infinite.
\end{lemma}
\begin{proof}
Fix $z \in Z$ and a set of parameters $A \subseteq \M$ such that $X$ and $\cY$ are $\cL(A)$-definable and $\rkL(a,b|A) = 2$.
Let $\alpha \in Y_z$ with $\rkL(\alpha|A,a,b) = 1$, so $\rkL(a,b|A,\alpha) = 2$.
Thus $\Inv_X(\alpha)_{a,b}$ is infinite by Lemma~\ref{lem:very_strong_JY}, so take $(c,d) \in \Inv_X(\alpha)_{a,b}$ with $\rkL(c,d|A,\alpha,a,b) = 1$.
Then $\rkL(c,d|A,a,b) = \rkL(c,d,\alpha|A,a,b)= 2$, so
\[
\rkL(c|A,a,b) = \rkL(d|A,a,b,c) = 1.
\]
The intersection $X \cap (a+(c-b)Y_z)$ is infinite, since it contains $d = a+(c-b)\alpha$.
Thus, $c \in X^\cY_{a,b,z}$, so this set is infinite as well.
\end{proof}

\subsection{Stable geometric fields are strongly minimal}

\begin{lemma}\label{lem:intersections}
Let $X\subseteq \M$ be an infinite, coinfinite definable set and let $Z\subseteq \M$ be infinite and definable.
Then there are $a\in \M$ and $b\in \M^\times$ such that $Z\cap(a+bX)$ and $Z\setminus(a+bX)$ are both infinite.
\end{lemma}
\begin{proof}
Choose a set $A$ of parameters over which $X$ and $Z$ are defined, and take elements $x\in X$, $y\in \M\setminus X$, and $z_1,z_2\in Z$ with $\rkL(z_1,z_2,x,y|A)=4$.
As the affine group acts $2$-transitively on $\M$, there are $a\in \M$ and $b\in \M^\times$ such that $a+bx=z_1$ and $a+by=z_2$.
We have
\[
\rkL(z_1,z_2|A,a,b)=\rkL(z_1,z_2,a,b|A)-\rkL(a,b|A)=2,
\]
so both $Z\cap(a+bX)$ and $Z\setminus(a+bX)$ are infinite, as they contain $z_1$ and $z_2$ respectively.
\end{proof}

\begin{thm}\label{thm:stablestronglyminimal}
If $T$ is stable, then it is strongly minimal.
\end{thm}
\begin{proof}
Suppose that $T$ is not strongly minimal and let $X \subseteq \M$ be an infinite and coinfinite definable set.
Let $\Aff(X)=(a+bX)_{a\in\M,\ b\in\M^\times}$ be the definable family of affine translates of $X$.
We show that $\Aff(X)$ has the binary tree property.

For each finite binary sequence $\nu\in 2^{<\omega}$, we choose $Y_\nu\in\Aff(X)$ as follows:
set $Y_\emptyset\coloneqq X$ and suppose that $Y_\eta$ has been defined for every $\eta\in 2^{\leq n}$ and that, for each $\nu\in 2^{n+1}$, the set
\[
Z_\nu\coloneqq \bigcap_{i\leq n,\,\nu(i)=0}(\M\setminus Y_{\nu|_i}) \cap \bigcap_{i\leq n,\,\nu(i)=1}Y_{\nu|_i}
\]
is infinite.
This holds when $n=0$, since both $X$ and $\M\setminus X$ are infinite.
For each $\nu\in 2^{n+1}$, apply Lemma~\ref{lem:intersections} to $X$ and $Z_\nu$ to choose $Y_\nu\in\Aff(X)$ such that $Z_\nu\cap Y_\nu$ and $Z_\nu\setminus Y_\nu$ are both infinite.
Doing this for all $\nu \in 2^{<\omega}$, we see that $\Aff(X)$ has the binary tree property, so $T$ is unstable; see~\cite[Definition 8.2.1 and Theorem 8.2.3]{TZ12}.
\end{proof}

Combining this result with~\cite[Theorem 1]{Hr92}, we immediately obtain:

\begin{corollary}\label{cor:stable_acf_constants}
Any algebraically bounded stable field is an expansion of an algebraically closed field by constants (see Section~\ref{sec:derivations} for a precise definition of algebraic boundedness).
\end{corollary}

We remark that we cannot relax the conditions here: by~\cite[Theorem 2 and Lemma 3]{Hr92}, there is a strongly minimal (thus, geometric) theory $T$ with two field structures of different characteristics (other similar fusion constructions exist as well).

\begin{remark}
The proof of Theorem~\ref{thm:stablestronglyminimal} given here replaces a longer, earlier proof and was found jointly with Erik Walsberg.
The second and third authors, together with Walsberg, use a variant of this argument in~\cite{KMW26} to give another proof of the result in~\cite{JTWY24} that large stable fields are separably closed.
\end{remark}

\subsection{Simple geometric fields are supersimple of SU-rank~1}
Let $\cY$ be a definable family and let $k \in \N^{>1}$.
We define $D^*(\cY,k)$ as follows: for $n \in \N$, we have $D^*(\cY,k) \geq n$ if we can find sets $(Y_{\eta})_{\emptyset \neq \eta \in \omega^{\leq n}}$ from $\cY$ such that
\begin{enumerate}
\item For all $\eta \in \omega^{<n}$, the family $(Y_{\eta, i})_{i <\omega}$ is $k$-inconsistent;
\item For all $\mu \in \omega^n$, the set $\bigcap_{\emptyset\neq \eta \trianglelefteq \mu} Y_{\eta}$ is infinite.
\end{enumerate}
In the intersection above, we use $\eta \trianglelefteq \mu$ to indicate that $\eta$ is an initial segment of $\mu$.
In the case $n = 0$, both conditions hold trivially, so $D^*(\cY,k)\geq 0$.
If $D^*(\cY,k)\geq n$ for all $n$, then $T$ is not simple.

\begin{thm}\label{thm:simple_supersimple_rank_1}
Suppose $T$ is simple.
Then $T$ is supersimple of SU-rank~1.
\end{thm}
\begin{proof}
We prove the contrapositive: let $\cY = (Y_z)_{z \in Z}$ be a definable family of subsets of $\M$ with each $Y_z$ infinite and suppose that some member of this family $k$-divides.
We will show that $T$ is not simple.
We consider two cases:

\textbf{Case 1:}
Suppose that for every infinite definable $X \subseteq \M$, every infinite sequence $(z_i)_{i<\omega}$ from $Z$, and every $\ell \in \N$, there are $a,b,c \in X$ such that the set $\{i<\omega:c\in X^\cY_{a,b,z_i}\}$ has cardinality at least $\ell$.
Consider the definable family
\[
\Aff(\cY)\coloneqq(a+bY_z)_{(a,b,z) \in \M \times \M^\times \times Z}
\]
We will show that $D^*(\Aff(\cY),k)\geq n$ for all natural numbers $n$ by induction.
Suppose we've already shown this for a given $n$, so we have sets $(\tilde{Y}_\eta)_{\emptyset \neq \eta \in \omega^{\leq n}}$ from $\Aff(\cY)$ such that
\begin{enumerate}
\item for $\eta \in \omega^{<n}$, the family $(\tilde{Y}_{\eta, i})_{i <\omega}$ is $k$-inconsistent, and
\item for all $\mu \in \omega^n$, the set $\bigcap_{\emptyset\neq \eta \trianglelefteq \mu} \tilde{Y}_{\eta}$ is infinite.
\end{enumerate}
Let $\mu \in \omega^n$ be given and set $X \coloneqq \bigcap_{\emptyset\neq \eta \trianglelefteq \mu} \tilde{Y}_{\eta}$.
Using our case assumption, the fact that some $Y_z$ $k$-divides, and saturation, we find a sequence $(z_i)_{i<\omega}$ from $Z$ and elements $a,b,c \in X$ such that $c\in X^\cY_{a,b,z_i}$ for each $i$ and $(Y_{z_i})_{i<\omega}$ is $k$-inconsistent.
We set $\tilde{Y}_{\mu,i} \coloneqq a+(c-b)Y_{z_i}$ for each $i$.
Then the family $(\tilde{Y}_{\mu,i})_{i<\omega}$ is $k$-inconsistent as well, and for each $i$, the intersection $X \cap \tilde{Y}_{\mu,i} = X \cap (a+(c-b)Y_{z_i})$ is infinite.
Defining $\tilde{Y}_{\mu,i}$ in this way for all $\mu\in \omega^n$ and all $i<\omega$, we see that $D^*(\Aff(\cY),k)\geq n+1$, as desired.

\textbf{Case 2:}
Suppose that there is an infinite definable set $X$, an infinite sequence $(z_i)_{i<\omega}$ from $Z$, and $\ell \in \N$ such that for all $(a,b,c) \in X^3$, the set $\{i<\omega:c\in X^\cY_{a,b,z_i}\}$ has cardinality $<\ell$.
Then for all $(a,b) \in X^2$, the family $(X^\cY_{a,b,z_i})_{i<\omega}$ is $\ell$-inconsistent.
Consider the definable family $\cX \coloneqq (X^\cY_{a,b,z})_{(a,b,z) \in X^2 \times Z}$.
We will show that $D^*(\cX, \ell) \geq n$ for all natural numbers $n$ by induction.
Suppose we've already shown this for a given $n$, so we have sets $(\tilde{X}_\eta)_{\emptyset \neq \eta \in \omega^{\leq n}}$ from $\cX$ such that
\begin{enumerate}
\item for $\eta \in \omega^{<n}$, the family $(\tilde{X}_{\eta, i})_{i <\omega}$ is $\ell$-inconsistent, and
\item for all $\mu \in \omega^n$, the set $\bigcap_{\emptyset\neq \eta \trianglelefteq \mu} \tilde{X}_{\eta}$ is infinite.
\end{enumerate}
Let $\mu \in \omega^n$ be given and let $W \coloneqq \bigcap_{\emptyset\neq \eta \trianglelefteq \mu} \tilde{X}_{\eta}$.
Let $(a,b) \in W^2$ be independent over the defining parameters for $W$ and $\cY$.
For each $i$, we set $\tilde{X}_{\mu,i}\coloneqq X^\cY_{a,b,z_i}$.
Then $(\tilde{X}_{\mu, i})_{i <\omega}$ is $\ell$-inconsistent, and for each $i$, we have
\[
W \cap \tilde{X}_{\mu,i} = W\cap X^\cY_{a,b,z_i} \supseteq W^\cY_{a,b,z_i},
\]
and the set $W^\cY_{a,b,z_i}$ is infinite by Lemma~\ref{lem:sufficiently_cover_implies_many_subsets}.
This shows that $D^*(\cX, \ell) \geq n+1$, as desired.
\end{proof}

Applying Fact~\ref{fact:supersimplecoincide}, we have:

\begin{corollary}\label{cor:simplecoincide}
Suppose that $T$ is simple.
Then $\find = \Tind\ $, so $SU(\av|A) = \rkL(\av|A)$ for $\av \in \M^n$.
\end{corollary}

\subsection{Rosy geometric fields are surgical}
The rosy case is very similar to the simple case, though some extra bookkeeping involving index sets is needed.
Let $\cY = (Y_z)_{z \in Z}$ and $\cZ = (Z_v)_{v \in V}$ be definable families in $\M^{\eq}$ where each $Z_v$ is a subset of $Z$.
Let $k \in \N^{>1}$.
We define $ \th^*(\cY,\cZ,k)$ as follows: for $n \in \N$, we have $\th^*(\cY,\cZ,k) \geq n$ if we can find sets $(Z_{\eta})_{\eta \in \omega^{<n}}$ from $\cZ$ and $(Y_{\eta})_{\emptyset \neq \eta \in \omega^{\leq n}}$ from $\cY$ such that
\begin{enumerate}
\item For all $\eta \in \omega^{<n}$, the sets $(Y_{\eta, i})_{i <\omega}$ are pairwise distinct members of the family $(Y_z)_{z \in Z_{\eta}}$, and this family is $k$-inconsistent;
\item For all $\mu \in \omega^n$, the set $\bigcap_{\emptyset\neq \eta \trianglelefteq \mu} Y_{\eta}$ is infinite.
\end{enumerate}
In the case $n = 0$, both conditions hold trivially, so $\th^*(\cY,\cZ,k)\geq 0$.

\begin{fact}
If there are $\cY$, $\cZ$, $k$ for which $\th^*(\cY,\cZ,k)\geq n$ for all $n$, then $T$ is not rosy.
\end{fact}
\begin{proof}
Let $\varphi(x,y)$ be the formula defining the family $\cY$ and let $\theta(y,z)$ be the formula defining $\cZ$.
Then $\th^*(\cY,\cZ,k) \leq \th(x=x,\varphi,\theta,k)$, where the latter is Onshuus' \thh-rank~\cite[Definition 3.1]{On06}; see also the proof of~\cite[Remark 3.1.2]{On06}.
Thus, if $\th^*(\cY,\cZ,k)\geq n$ for all $n$, then the same holds for $\th(x=x,\varphi,\theta,k)$, so $T$ is not rosy.
\end{proof}

\begin{thm}\label{thm:rosy_surgical}
Suppose $T$ is rosy.
Then $T$ is surgical.
\end{thm}
\begin{proof}
We prove the contrapositive: let $\cY = (Y_z)_{z \in Z}$ be a definable family of subsets of $\M$ with each $Y_z$ infinite and suppose that some member of this family \thh-divides.
We will show that $T$ is not rosy.
After shrinking $Z$, we may assume that $\cY$ is $k$-inconsistent for some $k$.
For $(a,b) \in \M^2$, we let $\tilde{Z}_{a,b} = \{(a,b)\} \times Z$.
We consider two cases:

\textbf{Case 1:} Suppose that for all infinite definable $X \subseteq \M$, there are $(a,b,c) \in X^3$ such that the set $\{z \in Z:c \in X^\cY_{a,b,z}\}$ is infinite.
Consider the definable families
\[
\Aff(\cY)\coloneqq(a+bY_z)_{(a,b,z) \in \M \times \M^\times \times Z},\qquad \cZ^\times \coloneqq (\tilde{Z}_{a,b})_{(a,b) \in \M\times \M^\times}.
\]
We will show that $\th^*(\Aff(\cY),\cZ^\times,k)\geq n$ for all natural numbers $n$ by induction.
Suppose we've already shown this for a given $n$, so we have sets $(\tilde{Z}_\eta)_{\eta \in \omega^{<n}}$ from $\cZ^\times$ as well as sets $(\tilde{Y}_{\eta})_{\emptyset \neq \eta \in \omega^{\leq n}}$ from $\Aff(\cY)$ where
\begin{enumerate}
\item for $\eta \in \omega^{<n}$, the sets $(\tilde{Y}_{\eta, i})_{i <\omega}$ are pairwise distinct members of the $k$-inconsistent family $(\tilde{Y}_w)_{w \in \tilde{Z}_{\eta}}$, and
\item for $\mu \in \omega^n$, the set $\bigcap_{\emptyset\neq \eta \trianglelefteq \mu} \tilde{Y}_{\eta}$ is infinite.
\end{enumerate}
Let $\mu \in \omega^n$ be given, and set $X \coloneqq \bigcap_{\emptyset\neq \eta \trianglelefteq \mu} \tilde{Y}_{\eta}$.
Using our case assumption, let $(a,b,c) \in X^3$ be such that the set $\{z \in Z:c \in X^\cY_{a,b,z}\}$ is infinite, and let $(z_i)_{i<\omega}$ be distinct elements of this set.
We set
\[
\tilde{Z}_\mu \coloneqq \tilde{Z}_{a,c-b},\qquad \tilde{Y}_{\mu, i} = a+(c-b)Y_{z_i} \text{ for each $i$}.
\]
Then the family $(\tilde{Y}_w)_{w \in \tilde{Z}_\mu} = (a+(c-b)Y_z)_{z \in Z}$ is $k$-inconsistent and the intersection
\[
X \cap \tilde{Y}_{\mu, i} = X \cap (a+(c-b)Y_{z_i})
\]
is infinite for each $i$.
Defining $(\tilde{Z}_\mu)_{\mu \in \omega^{n}}$ and $(\tilde{Y}_{\mu, i})_{(\mu,i) \in \omega^{n}\times \omega}$ in this way for all $\mu \in \omega^n$, we have shown that $\th^*(\Aff(\cY),\cZ^\times,k) \geq n+1$, as desired.

\textbf{Case 2:}
Now, suppose that we have an infinite definable $X \subseteq \M$ such that for all $(a,b,c) \in X^3$, the set $\{z \in Z:c \in X^\cY_{a,b,z}\}$ is finite.
As we are working in a saturated model, we obtain a natural number $\ell$ such that $\{z \in Z:c \in X^\cY_{a,b,z}\}$ has fewer than $\ell$ elements for all $(a,b,c) \in X^3$.
Thus, for all $(a,b) \in X^2$, the family $(X^\cY_{a,b,z})_{z \in Z}$ is $\ell$-inconsistent.
Consider the definable families
\[
\cX \coloneqq (X^\cY_{a,b,z})_{(a,b,z) \in X^2 \times Z},\qquad \cZ_X \coloneqq (\tilde{Z}_{a,b})_{(a,b) \in X^2}.
\]
This time, we will show that $\th^*(\cX,\cZ_X, \ell) \geq n$ for all natural numbers $n$ by induction.
Suppose we've already shown this for a given $n$, so we have sets $(\tilde{Z}_\eta)_{\eta \in \omega^{<n}}$ from $\cZ_X$ as well as sets $(\tilde{X}_{\eta})_{\emptyset \neq \eta \in \omega^{\leq n}}$ from $\cX$ where
\begin{enumerate}
\item for $\eta \in \omega^{<n}$, the sets $(\tilde{X}_{\eta, i})_{i <\omega}$ are pairwise distinct members of the $\ell$-inconsistent family $(\tilde{X}_w)_{w \in \tilde{Z}_{\eta}}$, and
\item for $\mu \in \omega^n$, the set $\bigcap_{\emptyset\neq \eta \trianglelefteq \mu} \tilde{X}_{\eta}$ is infinite.
\end{enumerate}
Let $\mu \in \omega^n$ be given, and let $W \coloneqq \bigcap_{\emptyset\neq \eta \trianglelefteq \mu} \tilde{X}_{\eta}$.
Let $(a,b) \in W^2$ be independent over the defining parameters for $W$ and $\cY$, and let $(z_i)_{i<\omega}$ be a tuple of distinct elements from $Z$.
We set
\[
\tilde{Z}_\mu \coloneqq \tilde{Z}_{a,b},\qquad \tilde{X}_{\mu,i}\coloneqq X^\cY_{a,b,z_i} \text{ for each $i$}.
\]
We've already seen that the family $(\tilde{X}_w)_{w \in \tilde{Z}_\mu} = (X^\cY_{a,b,z})_{z \in Z}$ is $\ell$-inconsistent.
For $i<\omega$, we have
\[
W \cap \tilde{X}_{\mu,i} = W\cap X^\cY_{a,b,z_i} \supseteq W^\cY_{a,b,z_i},
\]
and the set $W^\cY_{a,b,z_i}$ is infinite by Lemma~\ref{lem:sufficiently_cover_implies_many_subsets}.
Defining $(\tilde{Z}_\mu)_{\mu \in \omega^{n}}$ and $(\tilde{X}_{\mu, i})_{(\mu,i) \in \omega^{n}\times \omega}$ in this way for all $\mu \in \omega^n$, we have shown that $\th^*(\cX,\cZ_X, \ell) \geq n+1$, as desired.
\end{proof}

Applying Proposition~\ref{prop:surgicalrosy}, we have:

\begin{corollary}\label{cor:rosycoincide}
Suppose that $T$ is rosy.
Then $\thind = \Tind\ $, so $\Uth(\av|A) = \rkL(\av|A)$ for $\av \in \M^n$.
\end{corollary}

By the main theorem in~\cite{PP95} (see also~\cite[Proposition 4.1]{JY23}), we also have:

\begin{corollary}
Suppose that $T$ is rosy.
Then $\M$ has bounded Galois group---it has only finitely many extensions of degree $n$ for each $n\in \N$.
\end{corollary}

\begin{remark}
If $T$ is a simple theory in which forking and \thh-forking coincide, then one can deduce that $T$ is supersimple of SU-rank~1 directly from Theorem~\ref{thm:rosy_surgical}.
This holds if $T$ satisfies the stable forking conjecture~\cite{On06} or eliminates hyperimaginaries~\cite{Ea04}.
We are unsure whether Theorem~\ref{thm:simple_supersimple_rank_1} follows from Theorem~\ref{thm:rosy_surgical} without appealing to these major conjectures for simple theories.
\end{remark}

\begin{remark}\label{rem:kim}
The analogue of Theorems~\ref{thm:simple_supersimple_rank_1} and~\ref{thm:rosy_surgical} for Kim-independence does not hold, essentially because Kim forking does not satisfy base monotonicity in general.
The analogy would suggest that if a theory of fields $T$ is geometric and NSOP$_1$, then there is no non-algebraic single-variable formula $\varphi(x)$ that Kim-divides.
We will give an example contradicting this.
Let $T_E$ be the theory of an equivalence relation with infinitely many classes, each of which is infinite, in the language consisting of only a binary relation symbol $E$.
Winkler shows that the disjoint union of ACF and $T_E$ in the language $\cL_{\ring}\cup \{E\}$ has a model completion, which we denote by ACF$_E$~\cite{Wi75}.
This model companion is an \emph{interpolative fusion} in the sense of~\cite{KTW21}, and so it is algebraically bounded and has NSOP$_1$ by~\cite[Theorems 1.1(1) and 1.3(2)]{KTW22}.
Let $M \preceq \M \models \text{ACF}_E$ be a small elementary substructure and let $b \in \M \setminus M$ belong to a distinct $E$-class from each $a \in M$.
Then the formula $E(x,b)$ Kim-divides over $M$ in the reduct $T_E$, and so it Kim-divides over $M$ in ACF$_E$ by~\cite[Lemma 2.22]{KTW22}.
\end{remark}

\section{Adding derivations}\label{sec:derivations}
We assume for the rest of this note that $T$ extends the $\cL_{\ring}$-theory of fields of characteristic zero.
Let $\PP \coloneqq \dclL(\emptyset)$, and recall that $T$ is \textbf{algebraically bounded} if 
\[
a \in \aclL(B) \iff \trdeg(a|\PP(B)) = 0
\]
for all $a \in \M$ and $B \subseteq \M$.
We assume that one of the following holds:
\begin{enumerate}[(I)]
\item\label{I} $T$ is algebraically bounded, $\cL$ extends $\cL_{\ring}$ by relation symbols and constant
symbols for elements of $\PP$, and $T$ has quantifier elimination in the language $\cL$.
\item\label{II} $T$ is o-minimal and $\cL$ contains a function symbol for each $\cL(\emptyset)$-definable function.
Then $T$ has quantifier elimination and a universal axiomatization in the language $\cL$.
\end{enumerate}
In both cases, $T$ is geometric.
We let $T^{\forall}$ be the universal theory of $T$.
In case~\eqref{I}, models of $T^\forall$ are relational expansions of $\PP$-algebras.
In case~\eqref{II}, $T^\forall = T$.

\begin{lemma}\label{lem:Canonical}
In case~\eqref{I}, the independence relation $\Tind\ $ is real-canonical.
\end{lemma}
\begin{proof}
Let $A \subseteq \M$ and $C_1,C_2 \subseteq B \subseteq \M$ be given, and suppose $A \Tind_{C_i}B$ for $i = 1,2$.
We may assume that $C_i = \aclL(C_i)$ for $i= 1,2$ and that $B = \aclL(B)$.
Let $D \coloneqq C_1 \cap C_2$.
We need to show that $A \Tind_DB$.
Let $\av$ be a tuple from $A$ and assume that $\av$ is $\aclL$-dependent over $B$.
Then $\av$ is field-theoretically algebraically dependent over $B$, since $T$ is algebraically bounded.
Let $V$ be the Zariski locus of $\av$ over $B$ (the smallest Zariski closed set defined over $B$ containing $\av$).
Since $B$ is a relatively algebraically closed subfield of $\M$, $V$ is absolutely irreducible over $B$.
Let $B_0\subseteq B$ be the smallest field of definition for $V$.
Since $A \Tind_{C_i}B$, we have $B_0 \subseteq C_i$ for $i = 1,2$, so $B_0 \subseteq D$.
Thus $\rkL(\av|D) = \dim V = \rkL(\av|B)$.
\end{proof}

Note that the above lemma may fail in case~\eqref{II}.
Indeed, Pillay shows that for $T = \Th(\R,+,\cdot,\exp|_{[0,1]})$, the relation $\Tind\ $ is not real-canonical~\cite{Pi06}.

\begin{corollary}\label{cor:algboundedGEI}
Suppose $T$ is algebraically bounded and rosy.
Then $T$ has GEI.
\end{corollary}
\begin{proof}
By Corollary~\ref{cor:rosycoincide} and Lemma~\ref{lem:Canonical}, $\thind = \Tind\ $ is strict and real-canonical.
The corollary then follows from~\cite[Lemma 3.1]{Yo09}.
\end{proof}

\begin{definition}\label{def:Tderivation}
A \textbf{$T$-derivation on $K$} is a map $\delta\colon K\to K$ such that:
\begin{enumerate}
\item In case~\eqref{I}, $\delta$ is a derivation on $K$.
\item In case~\eqref{II}, $\delta$ satisfies the chain rule:
\[
\delta f(\uv) = \sum_{i = 1}^n\frac{\partial f}{\partial x_i}(\uv) \delta u_i
\]
for all $\uv = (u_1,\ldots,u_n) \in K^n$ and all function symbols $f \in \cL$ that are differentiable on an open neighborhood of $\uv$.
\end{enumerate}
\end{definition}
Let $\Delta = (\delta_1,\ldots,\delta_m)$ and let $\LD = \cL\cup \Delta$.
We let $\TDn$ be the
$\LD$-theory extending $T$ by axioms stating that $\Delta$ is a tuple of $T$-derivations, and we let $\TD$ further extend $\TDn$ by axioms stating that $\delta_1,\ldots,\delta_m$ all commute with each other.
We let $\Tb$ be one of $\TD,\TDn$.
In case~\eqref{I}, we ``hardcode'' the restriction of $\Delta$ to $\PP$ into the theory $\Tb$, as is done in~\cite{FT25}.
Given $K\models T^\forall$, an element $a\in K$, and a tuple of $T$-derivations $\Delta$ on $K$, we set $\Delta(a) \coloneqq (\delta_1a,\ldots,\delta_ma) \in K^m$.
We say that \textbf{$\Delta$ commutes at $a$} if $\delta_i\delta_ja = \delta_j\delta_ia$ for all $i,j$.

\begin{lemma}\label{lem:extend}
Let $K \subseteq_{\cL} M$ be models of $T^\forall$, let $\Delta_0$ be a tuple of $T$-derivations on $K$, and let $A \subseteq M$ be an $\aclL$-basis for $M$ over $K$.
Then for any map $s\colon A\to M^m$, there is a unique tuple of $T$-derivations $\Delta$ on $M$ that extends $\Delta_0$ and satisfies $s(a) = \Delta(a)$ for $a \in A$.
Moreover, if $\Delta_0$ commutes and $\Delta$ commutes at each $a \in A$, then $\Delta$ commutes.
\end{lemma}
\begin{proof}
In case~\eqref{I}, this is basic differential algebra; see~\cite[Fact 3.8]{FT25}.
In case~\eqref{II}, this follows by~\cite[Lemmas 2.13 and 6.1]{FK21}.
\end{proof}

\begin{corollary}\label{cor:amalgamation}
Let $M\models T$ and consider the following diagram of substructures:
\[
\begin{tikzcd}[cramped,row sep=small]
&(K_1;\Delta_1)\arrow[rd]&&&\\
(K_0;\Delta_0)\arrow[rd,swap,"\LD"]\arrow[ru,"\LD"]&&L\arrow[r]&\aclL(K_1,K_2)\arrow[r]&M\\
&(K_2;\Delta_2)\arrow[ru]&&&
\end{tikzcd}
\]
Above, unlabeled arrows denote $\cL$-substructures.
Suppose that $K_1\Tind_{K_0}K_2$.
Then there is a tuple of $T$-derivations $\Delta$ on $M$ extending $\Delta_1$ and $\Delta_2$ and making $M$ a model of $\TDn$.
Moreover
\begin{enumerate}
\item If $\Delta_1$ and $\Delta_2$ are both commuting tuples of $T$-derivations, then $\Delta$ can also be chosen to commute.
\item $L$ is closed under the $T$-derivations in $\Delta$, and the restriction of $\Delta$ to $L$ is the unique common extension of $\Delta_1$ and $\Delta_2$ to a tuple of $T$-derivations on $L$.
\end{enumerate}
\end{corollary}
\begin{proof}
Let $A_i \subseteq K_i$ be an $\aclL$-basis for $K_i$ over~$K_0$ for $i = 1,2$.
Then $A_1\cup A_2$ is an $\aclL$-basis for $L$ over $K_0$, since $K_1
\Tind_{K_0}K_2$.
Let $B$ be an $\aclL$-basis for $M$ over $\aclL(K_1,K_2)$, so
$A_1\cup A_2\cup B$ is an $\aclL$-basis for $M$ over $K_0$.
We extend $\Delta_0$ to a tuple
of $T$-derivations $\Delta$ on $M$ by taking $\Delta(a) = \Delta_i(a)$ for $a \in A_i$ and $\Delta(b)
= 0$ for $b \in B$.
The corollary follows from the previous lemma.
\end{proof}

Note that the above corollary gives precisely the conditions for being ``derivation-like'' in~\cite{LSM25}.

\begin{corollary}\label{cor:derivation-like}
The theories $\TDn$ and $\TD$ are derivation-like with respect to $(T,\Tind\ )$.
\end{corollary}

We now let $\TG$ denote the model completion of $\Tb$.
The existence of $\TG$ in case~\eqref{I} is established in~\cite{FT25} and existence of $\TDG$ in case~\eqref{II} is established in~\cite{FK21}.
The existence of $\TDnG$ in case~\eqref{II} is established in Appendix~\ref{sec:appendix} below.
Let $(\M;\Delta)$ be a monster model of $\TG$, so $\M$ is a monster model of $T$.
We write $\aclD$ in place of $\acl_{\LD}$, likewise with $\dclD$.
For $A \subseteq \M$, we set
\[
\Jet(A) \coloneqq \{\delta_{i_1}\delta_{i_2}\cdots \delta_{i_n}a: \text{$a \in A$, $n \in \N$, and $i_1,\ldots,i_n \in \{1,\ldots,m\}$}\},
\]
so $\Jet(A)$ is the closure of $A$ under the derivations in $\Delta$.
We write $\Jet(a)$ in place of $\Jet(\{a\})$.
We define the relation $\Dind\,$ on $(\M;\Delta)$ by
\[\textstyle
A\ \Dind_C\ B \iff \Jet(A) \Tind_{\Jet(C)} \Jet(B).
\]

In case~\eqref{I}, Fornasiero and Terzo show that $\Dind\,$ is an independence relation~\cite{FT25c}.
Their methods work just as well in case~\eqref{II}, using Corollary~\ref{cor:amalgamation} to establish extension.
An alternative proof of this fact can be obtained using Corollary~\ref{cor:derivation-like} and the results of Le\'on S\'anchez and Mohamed~\cite{LSM25}.

\begin{corollary}[{\cite[Theorem 6.1]{FT25c},~\cite[Theorem 3.12]{LSM25}}]\label{cor:dindrealstrict}
$\Dind\,$ is a real-strict independence relation.
If $\Tind\ $ is real-canonical with respect to $T$ (in particular, if we are in case~\eqref{I}), then $\Dind\,$ is real-canonical with respect to $\TG$.
\end{corollary}

The independence relation given in~\cite{LSM25} is defined using $\aclD$, and does not \emph{a priori} agree with the relation $\Dind\,$ above.
To establish this coincidence, one needs the following:

\begin{corollary}[{\cite[Corollary 6.2]{FT25c},~\cite[Lemma 3.8]{LSM25}, Corollary~\ref{cor:omindcl}}]\label{cor:aclD}
Let $A \subseteq (\M;\Delta)\models \TG$.
Then $\aclD(A) = \aclL(\Jet (A))$.
Consequently, we have
\[\textstyle
A\Dind_CB \iff \aclD(AC) \Tind_{\aclD(C)} \aclD(BC)
\]
for small sets $A,B,C\subseteq \M$.
\end{corollary}

We show below that $\Dind$ is actually the coarsest real-strict independence relation on $(\M;\Delta)$.
For this, we need a couple preliminary results.
Let $\ker(\Delta) \coloneqq \bigcap_{\delta\in \Delta}\ker(\delta)$ be the \textbf{field of absolute constants}.
Given $a \in \M$, we let $p_a^-(x)$ and $p_a^+(x)$ be the complete unary $\cL(\M)$-types determined by the sets of formulas $\{b<x<a:b \in \M^{<a}\}$ and $\{a<x<b:b \in \M^{>a}\}$, respectively.

\begin{lemma}\label{lem:invariantconstants}
Suppose we are in case~\eqref{II}, so $T$ is o-minimal.
Let $\ind$ be an independence relation on $(\M;\Delta)$, let $A,B,C \subseteq \M$ be small sets with $A\ind_CB$, and let $a \in A$.
Then there is $d \in \ker(\Delta)$ realizing $p_a^+|_{ABC}$ with $Ad\ind_CB$, and likewise for $p_a^-|_{ABC}$.
\end{lemma}
\begin{proof}
We prove the lemma for $p_a^+$, with the same proof working for $p_a^-$.
Full existence gives an $\LD(AC)$-automorphism $\sigma$ with $\sigma(B) \ind_{AC}B$.
Then $\sigma(B) \ind_CA$ as well and applying normality, transitivity, and symmetry, we get
\[\textstyle
ABC \ind_C\sigma(B),\qquad A\sigma(B)C \ind_C B.
\]
The axioms of $\TG$ tell us that there is a realization of $p_a^+|_{ABC}$ in $\ker(\Delta)$, likewise for $p_a^+|_{A\sigma(B)C}$.
Full existence allows us to find realizations $d_1 \models p_a^+|_{ABC}$ and $d_2 \models p_a^+|_{A\sigma(B)C}$ with 
\[\textstyle 
d_1\ind_{ABC}\sigma(B),\qquad d_2 \ind_{A\sigma(B)C}B.
\]
Applying normality, transitivity, and monotonicity to these relations and the ones above, we get
\[\textstyle
Ad_1\ind_C\sigma(B),\qquad Ad_2 \ind_C B.
\]
If $d_2\leq d_1$, then $d_2 \models p_a^+|_{ABC}$ as well, and we may take $d \coloneqq d_2$.
If $d_1 <d_2$, then applying $\sigma^{-1}$ gives 
\[\textstyle
\sigma^{-1}(d_1) \models \sigma^{-1}(p_a^+|_{A\sigma(B)C}) = p_a^+|_{ABC},\qquad A\sigma^{-1}(d_1) \ind_C B,
\]
so we may take $d \coloneqq \sigma^{-1}(d_1)$.
\end{proof}

\begin{corollary}\label{cor:smallneighborhood}
Suppose we are in case~\eqref{II}, let $\ind$ be an independence relation on $(\M;\Delta)$, let $A,B,C \subseteq \M$ be small sets with $A\ind_CB$, and let $\av \in A^n$ with $\rkL(\av|BC) = n$.
Then there is a finite $D\subseteq \ker(\Delta)$ and an $\cL(D)$-definable open $U \subseteq \M^n$ such that
\begin{enumerate}
\item $\rkL(D|ABC) = |D|$ and $AD\ind_CB$;
\item Any $\cL(BC)$-definable set $X \subseteq \M^n$ containing $\av$ also contains $U$.
\end{enumerate}
\end{corollary}
\begin{proof}
Using Lemma~\ref{lem:invariantconstants}, we inductively find a tuple $\dv = (d_1,\ldots,d_{2n}) \in \ker(\Delta)^{2n}$ with $A\dv \ind_CB$ such that 
\[
d_{2i-1} \models p_{a_i}^-|_{ABC\{d_j:j<2i-1\}} ,\qquad d_{2i} \models p_{a_i}^+|_{ABC \{d_j:j<2i\}}
\]
Let $U$ be the set $\prod_{i=1}^n(d_{2i-1},d_{2i}) \subseteq \M^n$.
Then $U$ is an infinitesimal (relative to $ABC$) neighborhood of $\av$, so any $\cL(BC)$-definable set $X \subseteq \M^n$ containing $\av$ also contains $U$.
\end{proof}

\begin{proposition}\label{prop:Dind-coarsest}
$\Dind\,$ is the coarsest real-strict independence relation on $(\M;\Delta)$, and so coincides with real \thh-forking independence.
\end{proposition}
Note that in case~\eqref{I}, this follows from Corollary~\ref{cor:dindrealstrict} and Remark~\ref{rem:canonicaliscoarsest}.
We give here another proof that doesn't use real-canonicity and unifies both cases.
\begin{proof}
Let $\ind$ be any real-strict independence relation on $(\M;\Delta)$.
By finite character and the definition of $\Dind\,$, it suffices to prove that for any small sets $A,B,C$ with $A\ind_{C}B$ and any finite tuple $\av$ from $\Jet(A)$, we have $\rkL(\av|\Jet(BC)) = \rkL(\av|\Jet(C))$.
We prove this by induction on $\rkL(\av|\Jet(BC))$.
Note that if $\av \in \aclL(\Jet(BC)) = \aclD(BC)$, then strong closure and real-strictness give $\av \in \aclD(C)$, so $\rkL(\av|\Jet(BC)) = \rkL(\av|\Jet(C)) = 0$.
We let $n>0$ and assume that our induction hypothesis holds for all small sets $A,B,C$ with $A\ind_CB$ and all tuples $\av$ from $\Jet(A)$, so long as $\rkL(\av|\Jet(BC))<n$.

Let $A,B,C$ be small sets with $A\ind_CB$ and let $\av$ be a tuple from $\Jet(A)$ with $\rkL(\av|\Jet(BC)) = n$.
We need to show that $\rkL(\av|\Jet(C)) = n$ as well.
By strong closure and Corollary~\ref{cor:aclD}, we may assume that $A,B,C$ are all $\aclD$-closed and that $C \subseteq A,B$, so $\rkL(\av|B) = n$.
Using full existence, take an $\LD(A)$-automorphism $\sigma$ of $(\M;\Delta)$ with $\sigma(B) \ind_AB$.
Since $C \subseteq A$, invariance and symmetry gives $\sigma(B)\ind_CA$, and transitivity and normality give $\sigma(B)\ind_CAB$.
Applying base monotonicity, monotonicity, and normality gives $B\sigma(B)\ind_B A$.
By symmetry and our induction hypothesis, we must have
\begin{equation}\label{eq:dim}
\rkL(\av|B\sigma(B)) = n.
\end{equation}
We also record here that applying monotonicity to $\sigma(B)\ind_CAB$ gives
\begin{equation}\label{eq:ind}\textstyle
\sigma(B)\ind_CB.
\end{equation}
We now consider the algebraically bounded and o-minimal cases separately.

\textbf{Case~\eqref{I}:}
Let $V\subseteq \M^n$ be the Zariski locus of $\av$ over $B$, so $V$ is absolutely irreducible over $B$ of dimension $n$.
Using that dimension and absolute irreducibility of Zariski-closed sets are definable in families~\cite[Appendix]{FLS17}, together with elimination of imaginaries in the theory of algebraically closed fields, we find a $\cL(C)$-definable family $(V_{\z})_{\z \in Z}$ of pairwise distinct absolutely irreducible varieties, all of the same dimension, and a tuple $\bv\in Z$ with coordinates from $B$ with $V_{\bv} = V$.
Note that $\av \in V_{\sigma(\bv)}$ as well, so $V_{\bv} \cap V_{\sigma(\bv)}$ has dimension $n$ by~\eqref{eq:dim}.
As the varieties in the family $(V_{\z})_{\z \in Z}$ are pairwise distinct and absolutely irreducible, we have $\bv = \sigma(\bv)$.
By~\eqref{eq:ind} and real-strictness, we have $\bv \in \aclD(C) = C$, so $\rkL(\av|C) = \rkL(\av|C\bv) = n$, as desired.

\textbf{Case~\eqref{II}:}
Let $\av'\in \M^n$ be a subtuple of $\av$ that is $\aclL$-independent over $B$.
Then $\av'$ is $\aclL$-independent over $B\sigma(B)$ as well by~\eqref{eq:dim}.
Since $A\ind_CB$ and $A\ind_BB\sigma(B)$, transitivity gives $A\ind_CB\sigma(B)$.
Applying Corollary~\ref{cor:smallneighborhood} to the sets $A,B\sigma(B),C$ and the tuple $\av'$, we find a finite set $D \subseteq \ker(\Delta)$ and an $\cL(D)$-definable open set $U \subseteq \M^n$ such that
\begin{enumerate}
\item $\rkL(D|AB) = |D|$ and $AD\ind_CB$;
\item Any $\cL(B\sigma(B))$-definable set $X \subseteq \M^n$ containing $\av'$ also contains $U$.
\end{enumerate}
As algebraic and definable closure coincide in the o-minimal case, we may write $\av = f(\av')$ for some $\cL(B)$-definable function $f$.
Take an $\cL(C)$-definable family $(f_{\z})_{\z \in Z}$ of functions, along with a tuple $\bv\in Z$ with coordinates in $B$ such that $f = f_{\bv}$.
Consider the relation $\sim$ on $Z$ given by 
\[
\z \sim \z' \iff f_{\z}|_U = f_{\z'}|_U.
\]
Then $\sim$ is $\cL(CD)$-definable, so definable choice yields an $\cL(CD)$-definable map $g\colon Z\to Z$ such that $\z \sim g(\z)$ and $\z \sim \z' \iff g(\z) = g(\z')$ for $\z,\z'\in Z$.
The tuple $\av'$ is contained in the $\cL(B\sigma(B))$-definable set
\[
\{\x\in \M^n: f_{\bv}(\x) = f_{\sigma(\bv)}(\x)\},
\]
so $U$ is contained in this set as well, giving $\bv\sim \sigma(\bv)$.
Let $\bv' \coloneqq g(\bv) = g(\sigma(\bv))$, so $f_{\bv'}(\av') = \av$.

Now starting with $AD\ind_CB\sigma(B)$ and using base monotonicity, monotonicity, and normality we get $\sigma(B)D\ind_{\sigma(B)} B$.
Using~\eqref{eq:ind} and applying transitivity yields $\sigma(B)D\ind_C B$.
Base monotonicity and normality give $\sigma(B)D \ind_{CD} BD$.
We have $\bv' \in \aclL(BD) \cap \aclL(\sigma(B)D)$, so $\bv' \ind_{CD}\bv'$ by normality.
Real-strictness gives $\bv' \in \aclD(CD) = \aclL(CD)$, since $\Jet(D) = D \cup \{0\}$.
Using that $\rkL(D|B) = \rkL(D|B\av) = |D|$, routine rank calculations give $\rkL(\av|BD) = \rkL(\av|B) = n$.
Likewise, $\rkL(\av|CD) = \rkL(\av|C)$.
Now we compute
\[
\rkL(\av|C) = \rkL(\av|CD) = \rkL(\av|CD\bv')= \rkL(\av|BD) = \rkL(\av|B) = n.\qedhere
\]
\end{proof}

Combining Remark~\ref{rem:GEIstrict}, Fact~\ref{fact:coincide}, and Proposition~\ref{prop:Dind-coarsest}, we have:

\begin{corollary}\label{cor:Dindcoincide}
Suppose that $\TG$ has GEI.
Then $\TG$ is rosy and $\Dind\,$ coincides with \thh-forking independence in $\TG$.
\end{corollary}
Note that the hypothesis of GEI is always satisfied in case~\eqref{II} by~\cite[Theorem 5.19]{FK21} and Corollary~\ref{cor:ominEI}.
For a different proof of Corollary~\ref{cor:Dindcoincide} in case~\eqref{I}, see {\cite[Corollary 6.2]{FT25c}}.
Combining~\cite[Theorem 3.14]{LSM25} with Corollary~\ref{cor:simplecoincide}, we also obtain another proof of~\cite[Theorem 7.1]{FT25c}:

\begin{corollary}\label{cor:simple}
If $T$ is simple, then so is $\TG$ and $\Dind\,$ coincides with forking independence in $\TG$.
\end{corollary}

The above corollary is, of course, vacuous in case~\eqref{II}, as is the next corollary.

\begin{corollary}\label{cor:NSOP1}
If $T$ is NSOP$_1$, then so is $\TG$.
Moreover, $\Dind\,$ is coarser than Kim-independence in $\TG$.
\end{corollary}
\begin{proof}
Suppose $T$ is NSOP$_1$, so we are necessarily in case~\eqref{I}.
Then in the terminology of~\cite{LSM25}, we have that $\Tind\ $ implies $\cL$-compositums, since $\cL$ extends $\cL_{\ring}$ by only relations and constant symbols, and that $T$ is very $\cL$-slim (which is the same as algebraic boundedness).
The corollary follows by~\cite[Corollary 3.19]{LSM25}.
\end{proof}

Finally, we note that using Theorem~\ref{thm:simple_supersimple_rank_1} and Corollary~\ref{cor:algboundedGEI}, we can improve~\cite[Theorems 7.5 and 7.9]{FT25c}.

\begin{corollary}
If $T$ is simple, then $\TG$ has GEI and eliminates imaginaries in the sorts of $T^{\eq}$.
\end{corollary}

\section{Ranks}\label{sec:ranks}
We denote the $\Dind\ $-rank by $U^\Delta$.
We are interested in when $\Dind\,$ is ranked.
We start by showing that we only need to consider the commuting case.
To do this, we need the following fact about $\TDnG$.

\begin{lemma}\label{lem:TDnGaxioms}
Let $(\M;\Delta) \models \TDnG$, let $a_1,\ldots,a_n \in \M$, and let $\theta_1,\ldots,\theta_n$ be pairwise incomparable elements of the free monoid generated by $\Delta$ (meaning that if $i \neq j$, then there is no $\gamma$ for which $\gamma\theta_i = \theta_j$).
Then, viewing each $\theta_i$ as an operator $\M\to \M$, there is $y \in \M$ such that $\theta_i y = a_i$ for each $i$.
\end{lemma}
\begin{proof}
Immediate from the \textbf{Deep} axioms; see~\cite[Section 4]{FT25} in case~\eqref{I} and the appendix in case~\eqref{II}.
\end{proof}

\begin{proposition}\label{prop:notranked}
Suppose that $|\Delta|\geq 2$.
Then $\TDnG$ is never supersimple or, more generally, superrosy.
Moreover, the independence relation $\Dind\,$ is not ranked.
\end{proposition}
\begin{proof}
Let $\epsilon\neq\delta \in \Delta$.
For each $n\geq0$, we define
\[
C_n \coloneqq \bigcap_{i\leq n}\ker(\epsilon\delta^i)
\]
Then $C_0 \supseteq C_1 \supseteq C_2 \supseteq \cdots$ is a descending chain of additive subgroups.
For a given $n$ and $a,b \in C_n$, we have $a+C_{n+1} = b+C_{n+1}$ if and only if $\epsilon\delta^{n+1}a = \epsilon\delta^{n+1}b$.
For any $d \in \M$, Lemma~\ref{lem:TDnGaxioms} gives $y \in\M$ with $\epsilon \delta^i y = 0$ for all $i\leq n$ and $\epsilon \delta^{n+1}y = d$.
Thus, $C_{n+1}$ has infinite index in $C_n$.
It follows that $\TDnG$ is not superrosy (in particular, not supersimple); see~\cite[Proposition 1.4]{EKP08}.
Let us show explicitly that $\Dind\,$ is not ranked.
Let $d_0 \coloneqq 0$ and, assuming we have defined $d_n$, take $d_{n+1}$ such that
\begin{enumerate}[(i)]
\item $\epsilon \delta^i(d_{n+1}-d_n) = 0$ for all $i\leq n$, and
\item $\epsilon\delta^{n+1}(d_{n+1})\not\in \aclD(d_0,\ldots,d_n)$.
\end{enumerate}
Again using Lemma~\ref{lem:TDnGaxioms}, we can always find such an element.
Note that $d_{n+1} \in \bigcap_{i\leq n}d_i+C_i$ for each $n$.
Let $p$ be a complete type over $\{d_i:i<\omega\}$ concentrating on $\bigcap_{i<\omega}(d_i+C_i)$, and for $n\geq 0$, let $p_n \coloneqq p|_{\{d_0,\ldots,d_n\}}$.
Fix $n$ and let $a \models p_{n+1}$.
Then $a \in d_{n+1} + C_{n+1}$, so
\[
\epsilon \delta^{n+1}a = \epsilon \delta^{n+1}d_{n+1} \not\in \aclD(d_0,\ldots,d_n) = \aclL(\Jet(d_0,\ldots,d_n)),
\]
It follows that
\[\textstyle
\Jet(a) \nind[\operatorname{M}]_{\Jet(d_0,\ldots,d_n)}\Jet(d_0,\ldots,d_{n+1}).
\]
so $p_{n+1}$ is an $\Dind\ $-forking extension of $p_n$.
Thus, $U^\Delta(p_0) = \infty$.
\end{proof}

\subsection{Bounding ranks via the Kolchin polynomial}
For the remainder of the section, $(\M;\Delta)$ is a monster model of $\TDG$.
We let $\preceq$ be the componentwise partial order on $\N^m$.
For $\uv = (u_1,\ldots,u_m) \in \N^m$, we put
\[
|\uv|\coloneqq u_1+\cdots +u_m,\qquad \delta^{\uv}\coloneqq \delta_1^{u_1}\cdots \delta_m^{u_m}\colon \M\to \M.
\]
For $\av = (a_1,\ldots,a_d) \in \M^d$, $\uv \in \N^m$, and $t \in \N$, we set
\[
\delta^\uv(\av) \coloneqq (\delta^{\uv}a_1,\ldots,\delta^{\uv}a_d),\qquad \Jet_{\leq t}(\av) \coloneqq \{\delta^{\uv}\av:|\uv|\leq t\},\qquad \Jet(\av) \coloneqq \{\delta^{\uv}\av:\uv \in \N^m\}.
\]
Then $\Jet(\av)$ agrees with our previous definition for singletons.
\begin{fact}[{\cite[Corollary A]{FK24}}]
Given $\av\in \M^d$ and a small subset $A\subseteq M$, there is a numerical polynomial $\omega_{\av| A}$, called the \textbf{Kolchin polynomial} of $\av$ over $A$, such that
\[
\omega_{\av| A}(t) = \rk_{\cL}(\Jet_{\leq t}(\av)|\Jet(A))\quad \text{for $t\gg0$}.
\]
\end{fact}
Above, $t\gg0$ abbreviates ``for all sufficiently large $t$.'' In case~\eqref{I}, $\rk_{\cL}$ coincides with the transcendence degree so the Kolchin polynomial is the classical Kolchin polynomial from~\cite{Ko73}.

Let $\cP$ be the set of all Kolchin polynomials, totally ordered by dominance, so
\[
P_1<_e P_2 \iff P_1(t)<P_2(t)\quad\text{for $t\gg0$.}
\]
By~\cite[Corollary 4.7]{FK24}, $\cP$ coincides with the set of Hilbert polynomials of finitely generated graded modules over a polynomial ring, which allows us to equip $\cP$ with an ordinal-valued ranking:

\begin{fact}[{\cite[Corollaries 3.2 and 3.15]{AP04}}]\label{fact:kolchinrankcomputation}
Each $P \in \cP$ can be written uniquely as a sum
\[
P(T) = \binom{T+a_1}{a_1}+\binom{T+a_2-1}{a_2}+\cdots+\binom{T+a_n-(n-1)}{a_n}
\]
for $n\geq 0$ and $a_1\geq a_2 \geq \cdots\geq a_n\geq 0$.
We have a strictly increasing map $\rk_{\cP}\colon \cP\to \On$ (the class of ordinals) where, for $P$ as above,
\[
\rk_\cP(P)\coloneqq \sum_{i\leq n}\omega^{a_i} \in \On \quad\text{(ordinal sum)}.
\]
\end{fact}

\begin{corollary}\label{cor:rankbounds}
Each nonzero $P \in \cP$ may be written
\[
P(T) = \frac{n}{k!}T^k+ \text{ lower-degree terms}.
\]
for some nonzero $n,k \in \N$.
For such a $P$, we have $\omega^kn\leq \rk_\cP(P)< \omega^k(n+1)$.
\end{corollary}

Let $p$ be an $\LD(A)$-type.
We set $\omega_p\coloneqq\omega_{\av|A}$, where $\av$ is some (any) realization of $p$.
Given a small set $B \supseteq A$, we always have $\omega_{\av| B} \leq_e\omega_{\av| A}$.
It follows that $\omega_{q} \leq_e \omega_{p}$ for any $\LD(B)$-type $q$ extending $p$.
Whether this inequality is strict is equivalent to $\Dind\ $-forking:

\begin{proposition}\label{prop:forkingextension}
Let $A \subseteq B$, let $p$ be an $\LD(A)$-type, and let $q$ be an $\LD(B)$-type extending $p$.
Then $q$ is an $\Dind\ $-forking extension of $p$ if and only if $\omega_{q} <_e \omega_{p}$.
\end{proposition}
\begin{proof}
Let $\av \models q$.
If $\av\Dind_A B$, then $\Jet(\av)\Tind_{\Jet(A)}\Jet(B)$, so $\rk_{\cL}(A_0|\Jet(B)) = \rk_{\cL}(A_0|\Jet(A))$ for all finite $A_0 \subseteq \Jet(\av)$.
In particular, $\rk_{\cL}(\Jet_{\leq t}(\av)|\Jet(B))=\rk_{\cL}(\Jet_{\leq t}(\av)|\Jet(A))$ for all $t$, so $\omega_{\av| B} =\omega_{\av| A}$.

Conversely, suppose $\av\nind[\Delta]_A B$ and take a finite $A_0 \subseteq \Jet(\av)$ with $\rk_{\cL}(A_0|\Jet(B)) < \rk_{\cL}(A_0|\Jet(A))$.
For $t\gg0$, we have $A_0 \subseteq \Jet_{\leq t}(\av)$, so $\rk_{\cL}(\Jet_{\leq t}(\av)|\Jet(B))<\rk_{\cL}(\Jet_{\leq t}(\av)|\Jet(A))$.
Thus $\omega_{\av| B} <_e \omega_{\av| A}$.
\end{proof}

\begin{corollary}\label{cor:rankedUdelta}
For any $\LD(A)$-type $p$, we have $U^\Delta(p) \leq \rk_{\cP}(\omega_p)$.
In particular, $\Dind\,$ is ranked.
\end{corollary}
Combining this with Corollaries~\ref{cor:Dindcoincide} and~\ref{cor:simple}, we get bounds on \thh-rank when $\TDG$ has GEI and on SU-rank when $T$ is simple.

To give more a more precise computation of $U^\Delta$, we make use of the \textbf{$\Delta$-closure operator} on $(\M;\Delta)$, denoted by $\clD$ and defined by
\[
a \in \clD(A) \iff \text{$\Jet(a)$ is not $\aclL$-independent over $\Jet(A)$.}
\]

\begin{fact}[{\cite[Theorem B]{FK24}}]\label{fact:Deltarankterm}
$(\M,\clD)$ is a pregeometry, and we have
\[
\omega_{\av| A}(T)= \frac{\rkD(\av|A)}{m!}T^m + \text{ lower-degree terms},
\]
where $\rkD$ is the rank function corresponding to the closure operator $\clD$.
If $\av$ is $\clD$-independent over $A$, then $\omega_{\av|A}(T) = \binom{T+m}{m}$.
\end{fact}

\begin{thm}\label{thm:Urank}
We have $U^\Delta(\TDG) = \omega^m$.
Moreover:
\begin{enumerate}
\item If $T$ is stable, then $\TDG$ is an expansion of DCF$_{0,m}$ by constants.
\item If $T$ is simple, then $\TDG$ is supersimple with $SU(\TDG) = \omega^m$.
\item If $\TDG$ has GEI, then $\TDG$ is superrosy with $\Uth(\TDG) = \omega^m$.
\end{enumerate}
\end{thm}
\begin{proof}
Let $a \in \M$.
Using Facts~\ref{fact:kolchinrankcomputation} and~\ref{fact:Deltarankterm}, we see that $\rk_{\cP}(\omega_{a|\emptyset}) \leq \omega^m$, with equality if and only if $a \not\in \clD(\emptyset)$.
By Corollary~\ref{cor:rankedUdelta}, this gives us an upper bound $U^\Delta(\TDG) \leq \omega^m$.
We now show that this is an equality.
For an ordinal $\eta <\omega^m$, take the unique tuple $\rv = (r_1,\ldots,r_m) \in \N^m$ such that $\eta$ has Cantor normal form $\sum_{i<m}\omega^ir_{i+1}$, and let $\delta^\eta$ denote the operator $\delta^\rv$.
Using saturation and the axioms of $\TDG$, we find $a \in \M \setminus \clD(\emptyset)$.
For $\eta <\omega^m$, set
\[
A_\eta\coloneqq \{\delta^\mu a:\mu\geq \eta\},\qquad p_\eta \coloneqq \tpD(a|A_\eta).
\]
Then $A_\eta = \Jet(A_\eta)$ and each $A_\eta$ is $\aclL$-independent.
We claim that $U^\Delta(p_\eta) \geq \eta$ for each $\eta$.
This is clear for $\eta = 0$ and follows by induction for $\eta$ a limit, so it suffices to show that $p_{\eta}$ is a $\Dind\ $-forking extension of $p_{\eta+1}$.
Fix $\eta$ and let $b \models p_\eta$, so
\[
\delta^\eta b = \delta^\eta a \in A_\eta \setminus \aclL(A_{\eta+1}).
\]
It follows that $b\nind[\Delta]_{A_{\eta+1}}A_\eta$, as desired.

To see (1), suppose $T$ is stable.
Then $T$ expands ACF$_0$ by constants by Corollary~\ref{cor:stable_acf_constants}.
By uniqueness of model completions, $\TDG$ must expand DCF$_{0,m}$, the theory of differentially closed fields with $m$ commuting derivations, by the same constants.
(2) and (3) are immediate by Corollaries~\ref{cor:Dindcoincide} and~\ref{cor:simple}.
\end{proof}

\subsection{Strongness}
An \textbf{inp-pattern of depth $\kappa$} is a collection of formulas $(\varphi_\lambda(x,\av_{\lambda,i}))_{\lambda<\kappa,i<\omega}$ (where $x$ is unary) such that:
\begin{enumerate}
\item For $\lambda<\kappa$, the collection $\{\varphi_{\lambda}(x,\av_{\lambda,i}):i<\omega\}$ is $k_\lambda$-inconsistent for some positive integer $k_\lambda$.
\item The partial type $\{\varphi_{\lambda}(x,\av_{\lambda,\eta(\lambda)}):\lambda<\kappa\}$ is consistent for all $\eta \colon \kappa \to \omega$.
\end{enumerate}
A theory has \textbf{burden $< \kappa$} if there is no inp-pattern of depth $\kappa$.
A \textbf{strong theory} is a theory of burden $<\omega$.
A theory that is strong and NIP is said to be \textbf{strongly dependent}.

\begin{lemma}\label{lem:infiniteformula}
Suppose we have an $\cL$-formula $\varphi(x,\y)$ with $x$ unary, a sequence $(\bv_i)_{i<\omega}$, and a positive integer $k$ such that $\varphi(x,\bv_i)$ is nonalgebraic for each $i$ and $\{\varphi(x,\bv_i):i<\omega\}$ is $k$-inconsistent.
Then $\TG$ is not strong.
\end{lemma}
\begin{proof}
Fix $\delta \in \Delta$ and for each $\lambda<\omega$, set $\varphi_\lambda(x,\y)\coloneqq \varphi(\delta^\lambda x,\y)$.
One easily verifies that $(\varphi_\lambda(x,\bv_i))_{\lambda,i<\omega}$ is an inp-pattern in $\TG$.
\end{proof}

\begin{proposition}\label{prop:notstrong}
If $|\Delta|\geq 2$, then $\TDnG$ is not strong.
\end{proposition}
\begin{proof}
Let $\epsilon\neq \delta \in \Delta$ and let $C \coloneqq \ker(\epsilon)$.
Then $C$ is an infinite index subgroup of the additive group $(\M,+)$, so take $(b_i)_{i<\omega}$ in distinct cosets of $C$.
Set
\[
\varphi_k(x,b_i) \coloneqq \epsilon(\delta^k x- b_i) = 0,
\]
so $(\M;\Delta)\models \varphi_k(a,b_i)$ if and only if $\delta^k a \in b_i+C$.
It follows from Lemma~\ref{lem:TDnGaxioms} that $(\varphi_k(x,b_i))_{k,i<\omega}$ is an inp-pattern in $\TDnG$.
\end{proof}

\begin{thm}\label{thm:strong}
$\TDG$ is strong (resp.\ strongly dependent) if and only if $T$ is simple (resp.\ stable).
Consequently, $\TDG$ is strong if and only if it expands DCF$_0$ by constants.
\end{thm}
\begin{proof}
If $T$ is not simple, then some formula $\varphi(x,\y)$ with $x$ unary has the tree property~\cite[Exercise 7.2.8]{TZ12}.
Taking a tree of parameters $(\av_{\eta})_{\eta \in \omega^{<\omega}}$, we see that the collection $\{\varphi(x,\av_{\langle i\rangle}):i<\omega\}$ is $k$-inconsistent for some $k$.
Thus, $\TDG$ is not strong by Lemma~\ref{lem:infiniteformula}.
Conversely, if $T$ is simple, then $\TDG$ is supersimple by Theorem~\ref{thm:Urank}.
Supersimple implies finite weight, which in turn implies that $\TDG$ is strong~\cite{Ad07,Wa00}.

To see that $\TDG$ is strongly dependent if and only if $T$ is stable, use that $\TDG$ is NIP if and only if $T$ is; see~\cite[Theorem 6.1]{FT25} for case~\eqref{I} and~\cite[Corollary 4.16]{FK21} for case~\eqref{II}.
Then these are in turn equivalent to $\TDG$ expanding DCF$_0$ by constants by Theorem~\ref{thm:Urank}.
\end{proof}

\section{Final remarks}
\begin{remark}
The inequality $U^\Delta(p) \leq \rk_{\cP}(\omega_p)$ in Corollary~\ref{cor:rankedUdelta} cannot, in general, be improved to an equality.
Indeed, given $n>0$, Suer gives an example in the theory of differentially closed fields with two commuting derivations of a type $p$ with $\rk_{\cP}(\omega_p) \geq \omega n$ but with $SU(p) = U^\Delta(p) = \omega$~\cite[Proposition 3.45]{Su07}.

Here is another example, following~\cite{Ma96}.
Suppose we are in case~\eqref{I} and that $\Delta$ consists of a single derivation $\delta$.
Take $a \in \M$ with
\[
\rkL(a, \delta a|\emptyset)= 2,\qquad \delta^2 a =\delta a/a.
\]
Let $p$ be the type of $a$ over $\emptyset$, so $\omega_p = 2$.
Let $q$ be a $\Dind\ $-forking extension of $p$ over a small differential subfield $K$, so $\omega_q \leq 1$.
We claim that $q$ is algebraic (that is, that $\omega_q = 0$).
Suppose that this is not the case.
Let $b \models q$ and let $P(X) \in K\{X\}$ be the minimal differential polynomial of $b$ over $K$.
Then $P$ has order 1, since $q$ is nonalgebraic, and $XX''-X'$ belongs to the differential ideal $I(P) \subseteq K\{X\}$.
By~\cite[Lemma 5.12]{Ma96}, $X'$ belongs to $I(P)$ as well, so $b' = 0$, contradicting the fact that no realization of $p$ is constant.

A potential advantage of working with derivations on o-minimal theories is that transcendental definable functions can sometimes be used to realize these ``missing'' intermediate Kolchin polynomials.
Indeed, suppose we are in case~\eqref{II}, let $a$ be as above, and assume that our o-minimal theory $T$ defines a logarithm on a neighborhood of $a$.
Then $\delta a =
\log(a)+c$ for some $c \in \ker(\delta) \setminus \dclL(\emptyset)$, and so $q \coloneqq \tpD(a|c)$ is a
nonalgebraic $\Dind\ $-forking extension of $p$.
\end{remark}

\begin{remark}
Let $\cL = \cL_{\ring}$ and let $T = \text{ACF}_0$.
Consider the theory DCF$_{0,m}$A, which is the model companion of the theory of fields with $m$ commuting derivations $\Delta = (\delta_1,\ldots,\delta_m)$ and a differential field automorphism $\sigma$ (so $\sigma$ is a field automorphism that also commutes with each $\delta_i$).
This theory was studied by Le\'on S\'anchez in his thesis~\cite{LS13}; see also~\cite{LS16}.
There, it is shown that DCF$_{0,m}$A is supersimple, with forking independence given by
\[\textstyle
A\find_CB \iff \Jet(A)\Tind_{\Jet(C)}\Jet(B)
\]
where here, $\Jet(A)$ is the closure of $A$ under each $\delta_i$, as well as $\sigma$ and $\sigma^{-1}$.
Let $\Jet^+(A)$ denote the closure of $A$ under each $\delta_i$ and $\sigma$ (but not necessarily $\sigma^{-1}$).
Then given any algebraic relation between elements of $\Jet(A)$ over $\Jet(B)$, we can apply $\sigma$ to this relation a sufficient number of times to obtain a relation between elements of $\Jet^+(A)$ over $\Jet(B)$.
That is, we have
\[\textstyle
A\find_CB \iff \Jet^+(A)\Tind_{\Jet(C)}\Jet(B).
\]
Now~\cite[Corollary A]{FK24} gives for any type $p = \tp(\av|A)$ in DCF$_{0,m}$A a numerical polynomial $\omega_p$ such that
\[
\omega_p(t) = \rkL(\Jet_{\leq t}(\av)|\Jet(A))\quad \text{for $t\gg0$},
\]
where $\Jet_{\leq t}(\av) = \{\delta^{\uv}\sigma^r(\av):(\uv,r) \in \N^{m+1},\ |\uv|+r\leq t\}$.
Then the proofs of Proposition~\ref{prop:forkingextension}, Corollary~\ref{cor:rankedUdelta} and Theorem~\ref{thm:Urank} go through in this setting to give $SU(p)\leq \rk_{\cP}(\omega_p)$; in particular, $SU(\text{DCF$_{0,m}$A}) = \omega^{m+1}$.
These bounds are implicit in~\cite[Lemma 5.3.5]{LS13}, but not explicitly stated.
In the case $m = 1$, something similar is done by Bustamante~Medina to show that DCFA $\coloneqq$ DCF$_{0,1}$A has SU-rank~$\omega^2$~\cite{BM11}.
\end{remark}

\appendix
\section{Extending independence to \texorpdfstring{$\M^{\eq}$}{Meq}}\label{sec:appendixa}
Let $T$ be a pregeometric theory.
Then $\acleq$ continues to satisfy exchange for real points over imaginary parameters~\cite[Lemma 3.1]{Ga05}, so $\rkL$ extends to a well-defined rank $\rkeq(\av|B)$ for any real tuple $\av$ and any set of parameters $B$, possibly imaginary.

Given sets $A,B \subseteq \M^{\eq}$, we call a real (possibly infinite) tuple $\bv$ a \textbf{basis} for $A$ over $B$ if $\bv$ is $\acleq$-independent over $B$ and $A \subseteq \acleq(B\bv)$.
Note that two bases need not be the same size, but if $A$ is finite, then it has a finite basis.
Following Gagelman, we further extend $\rkeq$ to finite imaginary tuples $\av$ by setting
\[
\rkeq(\av|B) \coloneqq |\bv|- \rkeq(\bv|B\av),
\]
where $\bv$ is a finite basis for $\av$ over $B$.

\begin{fact}[{\cite[Lemma 3.3 and Proposition 3.4]{Ga05}}]\label{fact:gagelman}
Suppose $T$ is pregeometric and let $\av$ and $B$ be as above.
\begin{enumerate}
\item\label{s1} For any finite real tuple $\bv$ with $\av \in \acleq(B\bv)$, we have 
\[
\rkeq(\av|B) = \rkeq(\bv|B)- \rkeq(\bv|B\av).
\]
In particular, $\rkeq(\av|B)$ does not depend on choice of basis.
\item\label{s2} There is a finite subset $B_0 \subseteq B$ with $\rkeq(\av|B_0) = \rkeq(\av|B)$.
\item\label{s3} $\rkeq(\av\bv|B) = \rkeq(\av|B\bv) + \rkeq(\bv|B)$ for any finite imaginary tuple $\bv$.
\item\label{s4} If $A\subseteq B$, then $\rkeq(\av|A)\geq\rkeq(\av|B)$.
\item\label{s5} $\rkeq(\av|B)$ only depends on $\tpL(\av|B)$.
\end{enumerate}
\end{fact}

Following Gagelman~\cite{Ga05}, we say that $T$ is \textbf{surgical} if it is pregeometric and if for every definable equivalence relation $E$ on a definable set $X$, only finitely many $E$-classes have the same $\aclL$-dimension as $X$.

\begin{fact}[{\cite[Proposition 3.5]{Ga05}}]\label{fact:gagelman2}
Let $T$ be pregeometric.
Then $T$ is surgical if and only if for all $a \in \M^{\eq}$ and all $B \subseteq \M^{\eq}$, we have $a\in \acleq(B) \iff \rkeq(a|B) = 0$.
\end{fact}

When $T$ is pregeometric, we use $\rkeq$ to define an independence relation $\eqind\ $ on small imaginary sets $A,B,C \subseteq \M^{\eq}$ as follows:
\[\textstyle
A\eqind_C B \iff \rkeq(A_0|BC) = \rkeq(A_0|C)\text{ for every finite subset }A_0 \subseteq A.
\]
As $\rkeq$ coincides with $\rkL$ for real elements, we see that $\eqind\ $ coincides with the independence relation $\Tind\ $ defined in Subsection~\ref{sec:ind-pregeom} for real sets.

\begin{proposition}\label{prop:eqindisind}
Suppose $T$ is pregeometric.
Then $\eqind\ $ is a real-strict independence relation on $\M^{\eq}$, the $\eqind\ $-rank of $\tpL(\av|B)$ coincides with $\rkeq(\av|B)$ for all finite tuples $\av$, and $\eqind\ $ is strict if and only if $T$ is surgical.
\end{proposition}
\begin{proof}
Finite character follows from the definition of $\eqind\ $, and invariance, monotonicity, base monotonicity, and normality follow from Fact~\ref{fact:gagelman}.

\emph{Local character:} First, suppose $A$ is finite, and let $\bv$ be a finite basis for $A$.
Take a finite set $C \subseteq B$ with $A \subseteq \acleq(C\bv)$.
By increasing $C$, we may assume that $\rkeq(\bv|BA) = \rkeq(\bv|CA)$.
Since $\rkeq(\bv|B) = \rkeq(\bv|C) = |\bv|$, we have $\rkeq(A|B) = \rkeq(A|C)$, so local character holds for $A$ with $\kappa(A) = \aleph_0$.
Now suppose that $A$ is infinite.
Let $(A_i)_{i \in I}$ enumerate the finite subsets of $A$, and for each $i$, use the finite case to take a finite subset $C_i \subseteq B$ such that $A_i \eqind_{C_i}B$.
Set $C\coloneqq \bigcup_{i \in I}C_i$, so $A_i \eqind_CB$ for each $i$ by base monotonicity.
Finite character gives that $A\eqind_CB$, so local character holds for $A$ with $\kappa(A) = |A|$.

\emph{Full existence:} Let $A,B,C$ be given and let $\bv$ be a basis for $A$ over $C$.
We extend $\tpL(\bv|C)$ to a partial type over $B$ by adding the $\cL(B)$-formula $\neg \varphi(\x)$ whenever $\varphi(\x)$ is an $\cL(B)$-formula asserting an algebraic relation between the coordinates of the variable $\x$.
Let $\bv^*$ be any tuple realizing this partial type, so $\bv^* \eqind_CB$.
Let $\sigma$ be an $\cL(C)$-automorphism of $\M$ mapping $\bv$ to $\bv^*$, and set $A^*\coloneqq \sigma(A)$, so $A^* \equiv_{\cL(C)}A$.
Let $A_0\subseteq A^*$ be finite, and take a finite basis $\bv_0 \subseteq \bv^*$ for $A_0$ over $C$.
We have
\[
\rkeq(A_0|B) = |\bv_0| - \rkeq(\bv_0|BA_0)\geq |\bv_0|- \rkeq(\bv_0|CA_0) = \rkeq(A_0|C)\geq \rkeq(A_0|B),
\]
where the last inequality uses Fact~\ref{fact:gagelman}\eqref{s4}.
We conclude that $A^*\eqind_CB$ by finite character.

\emph{Symmetry:} Suppose $A \eqind_CB$.
We need to show that $\rkeq(B_0|AC) = \rkeq(B_0|C)$ for a given finite $B_0\subseteq B$.
Using Fact~\ref{fact:gagelman}\eqref{s2}, take a finite subset $A_0 \subseteq A$ with $\rkeq(B_0|A_0C) = \rkeq(B_0|AC)$.
Take a finite basis $\bv$ for $A_0\cup B_0$ over $C$.
By assumption, 
\[
\rkeq(\bv|B_0C)- \rkeq(\bv|A_0B_0C) =\rkeq(A_0|B_0C)= \rkeq(A_0|C) =|\bv|- \rkeq(\bv|A_0C).
\]
Thus, 
\[
\rkeq(B_0|AC) = \rkeq(B_0|A_0C) = \rkeq(\bv|A_0C) - \rkeq(\bv|A_0B_0C) = |\bv| - \rkeq(\bv|B_0C) = \rkeq(B_0|C),
\]
 and we conclude that $B\eqind_CA$ by finite character.

\emph{Transitivity:} By symmetry, it suffices to show for $D \subseteq C \subseteq B$ that if $A \eqind_CB$ and $A\eqind_DC$, then $A\eqind_DB$.
Let $A_0 \subseteq A$ be finite; then $\rkeq(A_0|D) = \rkeq(A_0|C) = \rkeq(A_0|B)$, as needed.

\emph{Strictness and real-strictness:}
Suppose that $a \eqind_Ba$, so $\rkeq(a|B) = \rkeq(a|Ba) = 0$.
When $a$ is real, we have $a \in \acleq(B)$, and we can also conclude this for imaginary $a$ if and only if $T$ is surgical by Fact~\ref{fact:gagelman2}.

\emph{Rank equality:}
Let $\av$ be a tuple and $B$ be a parameter set, possibly both imaginary, and let $p \coloneqq \tpL(\av|B)$.
A straightforward induction on the $\eqind\ \,$-rank of $p$ shows that this rank is at most $\rkeq(\av|B)$.
To see that the $\eqind\ \,$-rank of $p$ is at least $\rkeq(\av| B)$, fix a finite basis $\bv = (b_1,\ldots,b_n)$ for $\av$ over $B$.
Let $m \coloneqq\rkeq(\av| B)$, so $\rkeq(\bv|B\av) = n-m$.
After rearranging, we may assume that $(b_1,\ldots,b_{n-m})$ are $\acleq$-independent over $B\av$.
For $i \leq m$, we set $\bv_i \coloneqq (b_1,\ldots,b_{n-m+i})$, and we calculate that 
\[
\rkeq(\av|B\bv_i) = \rkeq(\bv|B\bv_i) - \rkeq(\bv|B\av\bv_i) = m-i.
\]
Thus, $\tpL(\av|B\bv_i)$ is an $\eqind\ \,$-forking extension of $\tpL(\av|B\bv_{i-1})$ for $0<i\leq m$.
In particular, $p$ has $\eqind\ \,$-rank at least $m$.
\end{proof}

\begin{proposition}\label{prop:surgicalrosy}
The following are equivalent:
\begin{enumerate}
\item $T$ is surgical.
\item $T$ is rosy and pregeometric, and $\thind = \eqind\ $.
\item $T$ is superrosy of \thh-rank~1.
\end{enumerate}
\end{proposition}
\begin{proof}
\emph{(1)$\implies$(2):}
Suppose $T$ is surgical.
By Proposition~\ref{prop:eqindisind} and Fact~\ref{fact:coincide}, $T$ is rosy and $\thind$ is coarser than $\eqind\ $.
To establish equality, we first claim that $\Uth(\av|B) = \rkeq(\av|B)$ for any tuple $\av$ and small set $B$ in $\M^{\eq}$.
Note that $\Uth(\av|B)\leq \rkeq(\av|B)$ by Fact~\ref{fact:Uorder} and Proposition~\ref{prop:eqindisind}.
As \thh-independence is strict, we have $\Uth(a|B)= \rkeq(a|B)$ for $a\in \M$, and it follows from the Lascar inequalities that $\Uth(\av|B)= \rkeq(\av|B)$ for $\av \in \M^n$.
Now let $\av$ be an arbitrary finite tuple (possibly imaginary) and take a finite basis $\bv$ for $\av$ over $B$.
Then using the Lascar inequalities again, we obtain
\[
\rkeq(\av|B) = \rkeq(\bv|B)-\rkeq(\bv|Ba) = \Uth(\bv|B) - \Uth(\bv|B\av) = \Uth(\av|B).
\]
Now we fix $A,B,C \subseteq \M^{\eq}$ with $A\thind_CB$, and we need to show that $A\eqind_CB$.
Take a finite tuple $\av$ from $A$.
Then $\av \thind_CBC$ by symmetry and normality, so 
\[
\rkeq(\av|BC) = \Uth(\av|BC) = \Uth(\av|C)= \rkeq(\av|C),
\]
as desired.

\emph{(2)$\implies$(3):}
Immediate, since $T$ has $\eqind\ \,$-rank~1.

\emph{(3)$\implies$(1):} Our argument is quite similar to Gagelman's proof of Fact~\ref{fact:gagelman2}.
Suppose that $T$ is superrosy of \thh-rank~1.
It follows easily from the Lascar inequalities that $T$ is pregeometric.
As \thh-forking is strict and $\Uth(a|B)\leq 1$ for $a \in \M$, we have $\Uth(a|B) = \rkeq(a|B)$ for real singletons $a$.
Using the Lascar inequalities again, we see that $\Uth(\av|B) = \rkeq(\av|B)$ for all real tuples $\av$.
Now, suppose towards contradiction that we have an $\cL(B)$-definable set $X$ of dimension $d$ and an $\cL(B)$-definable equivalence relation $E$ on $X$ such that infinitely many $E$-classes have dimension $d$.
For $\av \in X$, we let $[\av]_E \in \M^{\eq}$ denote the equivalence class of $\av$.
Take an infinite sequence $(\av_i)_{i<\omega}$ in distinct $E$-classes with $\rkeq(\av_i|B[\av_i]_E) = d$.
By saturation, we may arrange that all $\av_i$ have the same $\cL(B)$-type.
Let $e\coloneqq [\av_0]_E$, so $\tpL(e|B)$ is non-algebraic and thus $\Uth(e|B)>0$.
The Lascar inequalities give $\Uth(\av_0|B) \geq \Uth(e|B) + \Uth(\av_0|Be)> d$, so $\rkeq(\av_0|B)>d$, a contradiction.
\end{proof}

\section{The model completion of \texorpdfstring{$\TDn$}{TΔ,nc} for o-minimal $T$}\label{sec:appendix}
Let $T$ be an o-minimal expansion of the theory of real closed ordered fields.
In this appendix, we show that $\TDn$ has a model completion, denoted $\TDnG$.
We then investigate some properties of $\TDnG$.

We maintain the same assumptions on $T$ from Section~\ref{sec:derivations}; namely, that $\cL$ contains function symbols for all $\cL(\emptyset)$-definable functions, so that $T$ has quantifier elimination and a universal axiomatization.
Given models $K \preceq M \models T$ and $A \subseteq M$, we let $K\langle A \rangle$ denote the substructure of $M$ generated by $K$ and $A$.
Then $K\langle A\rangle$ is a model of $T$ as well.
Given also an $\cL(K)$-definable set $X \subseteq K^n$, we write $X(M)$ for the definable subset of $M^n$ defined by the same formula as $X$.

Let $(K;\Delta) \models \TDn$.
We need the following fundamental lemma describing how $T$-derivations interact with $\cL(K)$-definable functions:
\begin{fact}[{\cite[Lemma 2.12]{FK21}}]\label{fact:basicTderivation}
Let $\delta \in \Delta$, let $k>0$, and let $f$ be an $\cL(K)$-definable $\cC^k$-function on an open set $U\subseteq K^n$.
Then there is a unique $\cL(K)$-definable $\cC^{k-1}$-function $f^{[\delta]}\colon U \to K$ such that for any $\TDn$-extension $(M;\Delta)\supseteq (K;\Delta)$ and any $\uv = (u_1,\ldots,u_n) \in U(M)$, we have
\[\textstyle
\delta f(\uv)= f^{[\delta]}(\uv)+\sum_{i=1}^n\frac{\partial f}{\partial x_i}(\uv) \delta u_i.
\]
\end{fact}

Let $\Gamma$ be the free monoid generated by $\Delta = (\delta_1,\ldots,\delta_m)$.
For $\Lambda\subseteq \Gamma$, we let $K^{\Lambda}$ consist of tuples $(z_\gamma)_{\gamma \in \Lambda}$ indexed by $\Lambda$.
When we say that $X \subseteq K^{\Lambda}$ is definable, we mean that there is a finite $\Lambda_0 \subseteq \Lambda$ and a definable $X_0 \subseteq K^{\Lambda_0}$ such that $X = X_0 \times K^{\Lambda \setminus \Lambda_0}$.

We view each $\gamma \in \Gamma$ as an operator $K\to K$, and given a subset $\Gamma_0 \subseteq \Gamma$ and a tuple $\av \in K^n$, we let
\[
\Gamma_0(\av)\coloneqq (\gamma a_i)_{(i,\gamma) \in \{1,\ldots,n\}\times \Gamma_0} \in K^{n\times \Gamma_0}.
\]
For $\gamma = \delta_{i_1}\delta_{i_2}\cdots \delta_{i_r} \in \Gamma$, we let $|\gamma|\coloneqq r$ be the \emph{order} of $\gamma$.
We partially order $\Gamma$ by
\[
\gamma_1\preceq \gamma_2 \iff \gamma_2 = \alpha\gamma_1 \text{ for some }\alpha \in \Gamma.
\]
Given a $\preceq$-downward closed subset $\Lambda \subseteq \Gamma$, we let $\Lambda^{\max}$ be the set of $\preceq$-maximal elements of $\Lambda$, and we let $\Lambda^\downarrow \coloneqq \Lambda \setminus \Lambda^{\max}$.

Following~\cite{FT25}, we consider the following axiom schemes:
\begin{description}
\item[Deep] For every finite $\preceq$-downward closed set $\Lambda \subseteq \Gamma$ and every $\cL(K)$-definable set $A \subseteq K^{\Lambda}$, if the projection $\pi_{\Lambda^\downarrow}(A)\subseteq K^{\Lambda^\downarrow}$ is open, then there is $a \in K$ with $\Lambda(a) \in A$.
\item[Wide] For every $n$ and every $\cL(K)$-definable set $X \subseteq K^n \times K^{n\times \Delta}$, if the projection $\pi_n(X)\subseteq K^n$ is open, then there is $\av \in K^n$ with $(\av,\Delta(\av)) \in X$.
\end{description}

\begin{lemma}\label{lem:wideext}
$(K;\Delta)$ has a $\TDn$-extension $(M;\Delta)$ that satisfies the \textbf{Wide} axiom scheme.
Any such extension also satisfies the \textbf{Deep} axiom scheme.
\end{lemma}
\begin{proof}
Let $X \subseteq K^n \times K^{n\times \Delta}$ be an $\cL(K)$-definable set with $\pi_n(X)\subseteq K^n$ open.
Take a $\dclL(K)$-independent tuple $\av$ in a $T$-extension of $K$ with $\av \in \pi_n X(K\langle \av\rangle)$.
Using~\cite[Lemma 2.13]{FK21}, we define a tuple of $T$-derivations $\Delta$ on $K\langle \av\rangle$ that extends those on $K$ and satisfies $(\av,\Delta(\av)) \in X(K\langle \av \rangle)$.

One deduces the \textbf{Deep} axiom scheme from the \textbf{Wide} axiom scheme via the standard method of rewriting a higher-order differential equation as a system of first-order differential equations; see~\cite[Lemma 4.6]{FT25} for details.
\end{proof}

\begin{remark}
Following the more delicate construction in~\cite[Lemma 4.2 and Proposition 4.3]{FK21}, we can define $T$-derivations that satisfy the \textbf{Wide} axioms on any $T$-extension $M$ of $K$ with $\rkL(M|K) = |M|\geq |T|$.
\end{remark}

Let $\TDnG$ extend $\TDn$ by the axiom scheme \textbf{Deep} given above.

\begin{thm}\label{thm:ncmodelcompanion}
$\TDnG$ is the model completion of $\TDn$.
The \textbf{Wide} axiom scheme is an alternative axiomatization of $\TDnG$ over $\TDn$.
\end{thm}
\begin{proof}
By Lemma~\ref{lem:wideext}, it suffices to show that the following diagram can be completed whenever $(K;\Delta) \subseteq (M;\Delta)$ are models of $\TDn$ and $(\M;\Delta) \models \TDnG$ is an $|M|^+$-saturated extension of $(K;\Delta)$.
\[
\begin{tikzcd}
(M;\Delta)\arrow[r,dashed,"\exists"]&(\M;\Delta) \models \TDnG\\
(K;\Delta)\arrow[u]\arrow[ru]&
\end{tikzcd}
\]
Let $a \in M \setminus K$.
We need to find $b \in \M$ with $\tpL(\Jet(b)|K) = \tpL(\Jet(a)|K)$, for then we can extend our inclusion $(K;\Delta)\subseteq (\M;\Delta)$ to an embedding $(K\langle \Jet(a)\rangle;\Delta)\to (\M;\Delta)$ and conclude by Zorn's lemma.
By saturation, it suffices to show that for every finite $\Gamma_0 \subseteq \Gamma$ and every $\cL(K)$-definable set $X \subseteq K^{\Gamma_0}$ with $\Gamma_0(a) \in X(M)$, there is $b \in \M$ with $\Gamma_0(b) \in X(\M)$.
One can show this following the proof of~\cite[Lemma 4.8]{FT25}, using obvious modifications based on~\cite{FK21}, but for completeness, we give the details below.

Put a total ordering $\leq$ on $\Gamma$, where
\[
\delta_{i_1}\cdots \delta_{i_k}\leq \delta_{j_1}\cdots \delta_{j_n} \iff \text{either $k<n$, or $k=n$ and $(i_1,\ldots,i_k) \leq_{lex}(j_1,\ldots,j_n)$}.
\]
Then $\leq$ refines $\preceq$, is left-invariant, and has order-type $\omega$.
Let
\[
\Gamma_a \coloneqq \big\{\gamma \in \Gamma: \gamma(a) \not\in K\langle \Gamma^{<\gamma}(a)\rangle\big\}.
\]
Then $\Gamma_a$ is $\preceq$-downward closed by Fact~\ref{fact:basicTderivation}.

Now, let $\Gamma_0 \subseteq \Gamma$ be finite and let $X\subseteq K^{\Gamma_0}$ be $\cL(K)$-definable with $\Gamma_0(a) \in X(M)$.
We may assume that $\Gamma_0$ is $\leq$-downward closed.
Let $\Lambda_0\coloneqq \Gamma_a \cap \Gamma_0$.
Then $\Lambda_0(a)$ is $\dclL$-independent over $K$, so there is an $\cL(K)$-definable open neighborhood $U\subseteq \pi_{\Lambda_0}(X)$ containing $\Lambda_0(a)$.
Let $\Lambda_1$ be the set of $\preceq$-minimal elements of $\Gamma_0 \setminus \Gamma_a$.
Then for each $\gamma \in \Lambda_1$, we have an $\cL(K)$-definable function $f_\gamma\colon M^{\Lambda_0}\to M$ with $\gamma(a) = f_\gamma(\Lambda_0(a))$.
Let $\Lambda \coloneqq \Lambda_0 \cup \Lambda_1$, let $x_\Lambda = (x_\gamma)_{\gamma \in \Lambda}$ be a tuple of variables, and consider the $\cL(K)$-definable set
\[
Y \coloneqq \{x_\Lambda \in K^\Lambda: x_{\Lambda_0} \in U\text{ and }x_\gamma = f_\gamma(x_{\Lambda_0})\}.
\]
Then $\Lambda(a) \in Y$, so applying Fact~\ref{fact:basicTderivation} and shrinking $U$, we can arrange that for any $y$ in any $\TDn$-extension $(L;\Delta)$ of $(K;\Delta)$, if $\Lambda(y) \in Y(L)$, then $\Gamma_0(y) \in X(L)$.
Thus, we need only find $b \in \M$ with $\Lambda(b) \in Y(\M)$.
Since $\Lambda^\downarrow \subseteq \Lambda_0$, the projection $\pi_{\Lambda^\downarrow}(Y)$ is open, so we find our desired $b \in \M$ using the \textbf{Deep} axioms.

By Lemma~\ref{lem:wideext}, the \textbf{Wide} axioms imply the \textbf{Deep} axioms over $\TDn$, and any model $(K;\Delta) \models\TDnG$ has a $\TDn$-extension $(M;\Delta)$ that satisfies the \textbf{Wide} axioms.
As $(K;\Delta)$ is existentially closed in $(M;\Delta)$ by the first part of the theorem, $(K;\Delta)$ also satisfies the \textbf{Wide} axioms.
\end{proof}

For the remainder of this appendix, let $(\M;\Delta) \models \TDnG$ be a monster model and let $B \subseteq \M$ be a small subset with $B = \Jet(B)$.
From the proof of Theorem~\ref{thm:ncmodelcompanion}, we see that for any $\av \in \M^n$, the type $\tpD(\av|B)$ is determined by $\tpL(\Jet(\av)|B)$.
This yields a description of definable sets:

\begin{corollary}\label{cor:definablesets}
For any $\LD(B)$-definable subset $X \subseteq \M^n$, there is a finite $\Gamma_0 \subseteq \Gamma$ and an $\cL(B)$-definable $\tilde{X} \subseteq \M^{n\times \Gamma_0}$ with
\[
X= \{\av \in \M^n: \Gamma_0(\av) \in \tilde{X}\}.
\]
\end{corollary}

Let $\ker(\Delta) \coloneqq \bigcap_{\delta\in \Delta}\ker(\delta)$ be the \textbf{field of absolute constants}.
The previous corollary allows us to describe the induced structure on $\ker(\Delta)$.

\begin{corollary}\label{cor:inducedconstant}
Every $\LD(B)$-definable subset of $\ker(\Delta)^n$ is of the form $\ker(\Delta)^n\cap A$ for some $\cL(B)$-definable subset $A \subseteq \M^n$.
\end{corollary}
\begin{proof}
Exactly as in the proof of~\cite[Lemma 4.23]{FK21}, using Corollary~\ref{cor:definablesets}.
\end{proof}

We can also lift combinatorial tameness from $T$ to $\TDnG$:

\begin{corollary}\label{cor:distalNIP}
The theory $\TDnG$ is distal (in particular, $\TDnG$ is NIP).
\end{corollary}
\begin{proof}
Exactly as in the proof of~\cite[Theorem 4.15]{FK21}, using Corollary~\ref{cor:definablesets}.
\end{proof}

\begin{thm}[O-minimal open core]\label{thm:opencore}
Any $\LD(B)$-definable open subset of $\M^n$ is $\cL(B)$-definable.
\end{thm}
\begin{proof}
We apply the characterization in~\cite[Corollary 3.1]{BH12}.
For a fixed $n$, we need to find $D \subseteq \M^n$ such that:
\begin{enumerate}
\item $D$ is dense in $\M^n$,
\item for every $\av \in D$ and open $U \subseteq \M^n$, if $\tpL(\av|B)$ is realized in $U$, then it is realized in $U\cap D$, and
\item for every $\av \in D$, the type $\tpD(\av|B)$ is implied by $\tpL(\av|B)$ in conjunction with ``$\av \in D$''.
\end{enumerate}
We claim that $D \coloneqq \ker(\Delta)^n$ realizes these three conditions.
For any open $\cL(\M)$-definable $U \subseteq \M^n$, the \textbf{Wide} axioms give $\av \in U$ with $\Delta(\av) = \zero$, so condition (1) holds.
Condition (3) follows immediately from Corollary~\ref{cor:inducedconstant}, so it remains to establish condition (2).
Let $\av = (a_1,\ldots,a_n) \in \ker(\Delta)^n$ be given and let $U \subseteq \M^n$ be an open set containing a realization of $\tpL(\av|B)$.
Let $k\leq n$ and suppose without loss of generality that $\av_0\coloneqq (a_1,\ldots,a_k)$ is a maximal $\dclL(B)$-independent subtuple of $\av$.
Take an $\cL(B)$-definable map $f = (f_1,\ldots,f_n)\colon\M^k\to \M^n$ with $f(\av_0) = \av$.
Since $f$ is generically continuous and $U$ contains a realization of $\tpL(\av|B)$, we find an open set $V \subseteq \M^k$ containing a realization of $\tpL(\av_0|B)$ with $f(V) \subseteq U$.
Since the set of realizations of $\tpL(\av_0|B)$ is open and $\ker(\Delta)^k$ is dense in $\M^k$, we find $\cv_0 \in \ker(\Delta)^k\cap V$ with $\cv_0 \models \tpL(\av_0|B)$.
Set $\cv \coloneqq f(\cv_0)$, so $\cv \in U$ and $\cv \models \tpL(\av|B)$.
We need to show that $\cv \in \ker(\Delta)^n$.
For $\delta \in \Delta$ and $j \in \{1,\ldots,n\}$, Fact~\ref{fact:basicTderivation} gives
\[
0= \delta a_j= \delta f_j(\av_0)= f_j^{[\delta]}(\av_0)+\sum_{i=1}^k\frac{\partial f_j}{\partial x_i}(\av_0) \delta a_i= f_j^{[\delta]}(\av_0).
\]
It follows that $f_j^{[\delta]}(\cv_0) = 0$ as well, so $\delta c_j = \delta f_j(\cv_0) = 0$, since $\cv_0 \in \ker(\Delta)^k$.
As this holds for all $\delta$ and all $j$, we have $\cv \in \ker(\Delta)^n$ as desired.
\end{proof}

Using Theorem~\ref{thm:opencore} and following the proofs of~\cite[Corollary 5.14 and Theorem 5.19]{FK21}, we obtain a characterization of $\dcl_{\LD}$ and elimination of imaginaries:

\begin{corollary}\label{cor:omindcl}
For any $A \subseteq \M$, we have $\dcl_{\LD}(A) = \dclL(\Jet(A))$.
In particular, the $\LD$-definably closed sets are exactly the models of $\TDn$.
\end{corollary}

\begin{corollary}\label{cor:ominEI}
The theory $\TDnG$ eliminates imaginaries.
\end{corollary}

\bibliographystyle{amsplain}	
\bibliography{trk}

@preamble{ {\providecommand{\noopsort}[1]{}} }

@phdthesis{Ad05,
	author = {Adler, Hans},
	school = {Albert-Ludwigs-Universit\"at Freiburg},
	title = {Explanation of Independence},
	year = {2005}}

@article {Ad09,
	author = {Adler, Hans},
	title = {A geometric introduction to forking and thorn-forking},
	journal = {J. Math. Log.},
	fjournal = {Journal of Mathematical Logic},
	volume = {9},
	year = {2009},
	number = {1},
	pages = {1--20},
	doi = {10.1142/S0219061309000811}}

@article{FK21,
	author = {Fornasiero, Antongiulio and Kaplan, Elliot},
	title = {Generic derivations on o-minimal structures},
	journal = {J. Math. Log.},
	fjournal = {Journal of Mathematical Logic},
	volume = {21},
	year = {2021},
	number = {2},
	pages = {Paper No. 2150007, 45},
	doi = {10.1142/S0219061321500070}}

@article{McG00,
	author = {McGrail, Tracey},
	fjournal = {Journal of Symbolic Logic},
	journal = {J. Symbolic Logic},
	number = {2},
	pages = {885--913},
	title = {The model theory of differential fields with finitely many commuting derivations},
	volume = {65},
	year = {2000},
	doi = {10.2307/2586576}}

@article{On06,
	author = {Onshuus, Alf},
	fjournal = {The Journal of Symbolic Logic},
	journal = {J. Symbolic Logic},
	number = {1},
	pages = {1--21},
	title = {Properties and consequences of thorn-independence},
	volume = {71},
	year = {2006},
	doi = {10.2178/jsl/1140641160}}

@article {EO07,
	author = {Ealy, Clifton and Onshuus, Alf},
	title = {Characterizing rosy theories},
	journal = {J. Symbolic Logic},
	fjournal = {The Journal of Symbolic Logic},
	volume = {72},
	year = {2007},
	number = {3},
	pages = {919--940},
	doi = {10.2178/jsl/1191333848}}

@article{FT25,
	author = {Fornasiero, Antongiulio and Terzo, Giuseppina},
	fjournal = {The Journal of Symbolic Logic},
	journal = {J. Symbolic Logic},
	title = {Generic derivations on algebraically bounded structures},
	year = {2024},
	note = {to appear},
	doi = {10.1017/jsl.2024.57}}

@article{FT25c,
	author = {Fornasiero, Antongiulio and Terzo, Giuseppina},
	title = {Generic derivations on algebraically bounded structures {II}:\ Model theoretical properties},
	eprint = {2507.22181},
	note = {preprint. \url{https://arxiv.org/abs/2507.22181}},
	year = {2025}}

@article {dE23,
	author = {d'Elb\'{e}e, Christian},
	eprint = {2308.07064},
	note = {preprint. \url{https://arxiv.org/abs/2308.07064}},
	title = {Axiomatic Theory of Independence Relations in Model Theory},
	year = {2023}}

@article{LSM25,
	author = {Le\'on S\'anchez, Omar and Mohamed, Shezad},
	title = {Neostability transfers in derivation-like theories},
	fjournal = {Model Theory},
	journal = {Model Theory},
	number = {2},
	pages = {177--201},
	volume = {4},
	year = {2025},
	doi = {10.2140/mt.2025.4.177}}

@article{Pi06,
	author = {Pillay, Anand},
	title = {Canonical bases in o-minimal and related structures},
	note = {preprint},
	year = {2006}}

@article {EKP08,
	author = {Ealy, Clifton and Krupi\'nski, Krzysztof and Pillay, Anand},
	title = {Superrosy dependent groups having finitely satisfiable generics},
	journal = {Ann. Pure Appl. Logic},
	fjournal = {Annals of Pure and Applied Logic},
	volume = {151},
	year = {2008},
	number = {1},
	pages = {1--21},
	doi = {10.1016/j.apal.2007.09.004}}

@article{FK24,
	author = {Fornasiero, Antongiulio and Kaplan, Elliot},
	title = {Hilbert polynomials for finitary matroids},
	fjournal = {Pacific Journal of Mathematics},
	journal = {Pacific J. Math.},
	number = {2},
	pages = {273--308},
	volume = {333},
	year = {2024},
	doi = {https://doi.org/10.2140/pjm.2024.333.273}}

@article {AP04,
	author = {Aschenbrenner, Matthias and Pong, Wai Yan},
	title = {Orderings of monomial ideals},
	journal = {Fund. Math.},
	fjournal = {Fundamenta Mathematicae},
	volume = {181},
	year = {2004},
	number = {1},
	pages = {27--74},
	doi = {10.4064/fm181-1-2}}

@article{BH12,
	author = {Boxall, Gareth and Hieronymi, Philipp},
	fjournal = {Journal of Symbolic Logic},
	journal = {J. Symbolic Logic},
	number = {1},
	pages = {111--121},
	title = {Expansions which introduce no new open sets},
	volume = {77},
	year = {2012},
	doi = {10.2178/jsl/1327068694}}

@incollection{Ma96,
	author = {Marker, David},
	title = {Model Theory of Differential Fields},
	booktitle = {Model theory of fields},
	series = {Lecture Notes in Logic},
	volume = {5},
	publisher = {Springer-Verlag, Berlin},
	year = {1996},
	pages = {38--113},
	doi = {10.1007/978-3-662-22174-7}}

@article {LS16,
	author = {Le\'on S\'anchez, Omar},
	title = {On the model companion of partial differential fields with an automorphism},
	journal = {Israel J. Math.},
	fjournal = {Israel Journal of Mathematics},
	volume = {212},
	year = {2016},
	number = {1},
	pages = {419--442},
	doi = {10.1007/s11856-016-1292-y}}

@article {BM11,
	author = {Bustamante Medina, Ronald F.},
	title = {Rank and dimension in difference-differential fields},
	journal = {Notre Dame J. Form. Log.},
	fjournal = {Notre Dame Journal of Formal Logic},
	volume = {52},
	year = {2011},
	number = {4},
	pages = {403--414},
	doi = {10.1215/00294527-1499363}}

@phdthesis{LS13,
	author = {Le\'on S\'anchez, Omar},
	school = {University of Waterloo},
	title = {Contributions to the model theory of partial differential fields},
	year = {2013}}

@phdthesis{Su07,
	author = {Suer, Sonat},
	school = {University of Illinois at Urbana-Champaign},
	title = {Model theory of differentially closed fields with several commuting derivations},
	year = {2007}}

@book {Wa00,
	author = {Wagner, Frank O.},
	title = {Simple theories},
	series = {Mathematics and its Applications},
	volume = {503},
	publisher = {Kluwer Academic Publishers, Dordrecht},
	year = {2000},
	pages = {xii+260},
	doi = {10.1007/978-94-017-3002-0}}

@article {JY23,
	author = {Johnson, Will and Ye, Jinhe},
	title = {A note on geometric theories of fields},
	journal = {Model Theory},
	fjournal = {Model Theory},
	volume = {2},
	year = {2023},
	number = {1},
	pages = {121--132},
	doi = {10.1007/s44007-022-00041-y}}

@article {Hr92,
	author = {Hrushovski, Ehud},
	title = {Strongly minimal expansions of algebraically closed fields},
	journal = {Israel J. Math.},
	fjournal = {Israel Journal of Mathematics},
	volume = {79},
	year = {1992},
	number = {2-3},
	pages = {129--151},
	doi = {10.1007/BF02808211}}

@incollection {Hr02,
	author = {Hrushovski, Ehud},
	title = {Pseudo-finite fields and related structures},
	booktitle = {Model theory and applications},
	series = {Quad. Mat.},
	volume = {11},
	pages = {151--212},
	publisher = {Aracne, Rome},
	year = {2002}}

@article {CP98,
	author = {Chatzidakis, Zo\'{e} and Pillay, Anand},
	title = {Generic structures and simple theories},
	journal = {Ann. Pure Appl. Logic},
	fjournal = {Annals of Pure and Applied Logic},
	volume = {95},
	year = {1998},
	number = {1-3},
	pages = {71--92},
	doi = {10.1016/S0168-0072(98)00021-9},
	url = {https://doi.org/10.1016/S0168-0072(98)00021-9}}

@article {Kr15,
	author = {Krupi\'nski, Krzysztof},
	title = {Superrosy fields and valuations},
	journal = {Ann. Pure Appl. Logic},
	fjournal = {Annals of Pure and Applied Logic},
	volume = {166},
	year = {2015},
	number = {3},
	pages = {342--357},
	doi = {10.1016/j.apal.2014.11.006},
	url = {https://doi.org/10.1016/j.apal.2014.11.006}}

@article {FT25b,
	author = {Fornasiero, Antongiulio and Terzo, Giuseppina},
	title = {Exponential fields: lack of generic derivations},
	journal = {Notre Dame J. Form. Log.},
	fjournal = {Notre Dame Journal of Formal Logic},
	volume = {66},
	year = {2025},
	number = {4},
	pages = {513--519},
	doi = {10.1215/00294527-2025-0014},
	url = {https://doi.org/10.1215/00294527-2025-0014}}

@article{Po25,
	author = {Point, Fran\c{c}oise},
	title = {On definable groups in dp-minimal topological fields equipped with a generic derivation},
	eprint = {2505.07044},
	note = {preprint, \url{https://arxiv.org/abs/2505.07044}},
	year = {2025}}

@article{CKP23,
	author = {Cubides Kovacsics, Pablo and Point, Fran\c{c}oise},
	title = {Topological fields with a generic derivation},
	journal = {Ann. Pure Appl. Logic},
	fjournal = {Annals of Pure and Applied Logic},
	volume = {174},
	year = {2023},
	number = {3},
	pages = {Paper No. 103211},
	doi = {10.1016/j.apal.2022.103211},
	url = {https://doi.org/10.1016/j.apal.2022.103211}}

@article{KK26,
	author = {Kaplan, Elliot and Kesting, Christoph},
	title = {Generic derivations, differential largeness, and {NTP}$_2$},
	journal = {Ann. Pure Appl. Logic},
	fjournal = {Annals of Pure and Applied Logic},
	volume = {177},
	number = {6},
	pages = {Paper No. 103730},
	year = {2026},
	doi = {10.1016/j.apal.2026.103730},
	url = {https://doi.org/10.1016/j.apal.2026.103730}}

@article {Ga05,
	author = {Gagelman, Jerry},
	title = {Stability in geometric theories},
 	journal = {Ann. Pure Appl. Logic},
	fjournal = {Annals of Pure and Applied Logic},
	volume = {132},
	year = {2005},
	number = {2-3},
	pages = {313--326},
	doi = {10.1016/j.apal.2004.10.016},
	url = {https://doi.org/10.1016/j.apal.2004.10.016}}

@book{TZ12,
	author = {Tent, Katrin and Ziegler, Martin},
	publisher = {Association for Symbolic Logic, La Jolla, CA; Cambridge University Press, Cambridge},
	series = {Lecture Notes in Logic},
	title = {A course in model theory},
	volume = {40},
	year = {2012},
	doi = {10.1017/CBO9781139015417},
	url = {https://doi.org/10.1017/CBO9781139015417}}

@article{Ad07,
	author = {Adler, Hans},
	title = {Strong theories, burden, and weight},
	note = {preprint},
	year = {2007}}

@book{Ko73,
	author = {Kolchin, E. R.},
	publisher = {Academic Press, New York-London},
	series = {Pure and Applied Mathematics},
	title = {Differential algebra and algebraic groups},
	volume = {54},
	year = {1973}}

@article {CHRK22,
	author = {Cluckers, Raf and Halupczok, Immanuel and Rideau-Kikuchi, Silvain},
	title = {Hensel minimality {I}},
	journal = {Forum Math. Pi},
	fjournal = {Forum of Mathematics. Pi},
	volume = {10},
	year = {2022},
	pages = {Paper No. e11, 68},
	doi = {10.1017/fmp.2022.6},
	url = {https://doi.org/10.1017/fmp.2022.6}}

@article{vdD89,
	author = {{\noopsort{Dries}}{van den Dries}, Lou},
	fjournal = {Annals of Pure and Applied Logic},
	journal = {Ann. Pure Appl. Logic},
	number = {2},
	pages = {189--209},
	title = {Dimension of definable sets, algebraic boundedness and {H}enselian fields},
	volume = {45},
	year = {1989},
	doi = {10.1016/0168-0072(89)90061-4},
	url = {https://doi.org/10.1016/0168-0072(89)90061-4}}

@article {FLS17,
	author = {Freitag, James and Li, Wei and Scanlon, Thomas},
	title = {Differential {C}how varieties exist},
	note = {With an appendix by William Johnson},
	journal = {J. Lond. Math. Soc. (2)},
	fjournal = {Journal of the London Mathematical Society. Second Series},
	volume = {95},
	year = {2017},
	number = {1},
	pages = {128--156},
	doi = {10.1112/jlms.12002},
	url = {https://doi.org/10.1112/jlms.12002}}

@article{Yo09,
	author = {Yoneda, Ikuo},
	title = {Some remarks on {CM}-triviality},
	volume = {61},
	journal = {J. Math. Soc. Japan},
	fjournal = {Journal of the Mathematical Society of Japan},
	number = {2},
	publisher = {Mathematical Society of Japan},
	pages = {379--391},
	year = {2009},
	doi = {10.2969/jmsj/06120379},
	url = {https://doi.org/10.2969/jmsj/06120379}}

@article{DMS10,
	author = {Dolich, Alfred and Miller, Chris and Steinhorn, Charles},
	fjournal = {Transactions of the American Mathematical Society},
	journal = {Trans. Amer. Math. Soc.},
	number = {3},
	pages = {1371--1411},
	title = {Structures having o-minimal open core},
	volume = {362},
	year = {2010},
	doi = {10.1090/S0002-9947-09-04908-3},
	url = {https://doi.org/10.1090/S0002-9947-09-04908-3}}

@article{PS86,
	author = {Pillay, Anand and Steinhorn, Charles},
	fjournal = {Transactions of the American Mathematical Society},
	journal = {Trans. Amer. Math. Soc.},
	number = {2},
	pages = {565--592},
	title = {Definable sets in ordered structures. {I}},
	volume = {295},
	year = {1986},
	doi = {10.2307/2000052},
	url = {https://doi.org/10.2307/2000052}}

@article {Jo24,
	author = {Johnson, Will},
	eprint = {2403.17478},
	note = {preprint, \url{https://arxiv.org/abs/2403.17478}},
	title = {${C}$-minimal fields have the exchange property},
	year = {2024}}

@article {CHRKV23,
	author = {Cluckers, Raf and Halupczok, Immanuel and Rideau-Kikuchi, Silvain and Vermeulen, Floris},
	title = {Hensel minimality {II}: {M}ixed characteristic and a diophantine application},
	journal = {Forum Math. Sigma},
	fjournal = {Forum of Mathematics. Sigma},
	volume = {11},
	year = {2023},
	pages = {Paper No. e89, 33},
	doi = {10.1017/fms.2023.91},
	url = {https://doi.org/10.1017/fms.2023.91}}

@article{Fo11,
	author = {Fornasiero, Antongiulio},
	fjournal = {Annals of Pure and Applied Logic},
	journal = {Ann. Pure Appl. Logic},
	number = {7},
	pages = {514--543},
	title = {Dimensions, matroids, and dense pairs of first-order structures},
	volume = {162},
	year = {2011},
	doi = {10.1016/j.apal.2011.01.003},
	url = {https://doi.org/10.1016/j.apal.2011.01.003}}

@phdthesis{Ea04,
	author = {Ealy, Clifton},
	school = {University of California, Berkeley},
	title = {Thorn forking in simple theories and a {M}anin-{M}umford theorem for {T}-modules},
	year = {2004}}

@article {KTW21,
	author = {Kruckman, Alex and Tran, Chieu-Minh and Walsberg, Erik},
	title = {Interpolative fusions},
	journal = {J. Math. Log.},
	fjournal = {Journal of Mathematical Logic},
	volume = {21},
	year = {2021},
	number = {2},
	pages = {Paper No. 2150010, 38},
	doi = {10.1142/S0219061321500100},
	url = {https://doi.org/10.1142/S0219061321500100}}

@article{KTW22,
	author = {Kruckman, Alex and Tran, Chieu-Minh and Walsberg, Erik},
	title = {Interpolative fusions {II}:\ Preservation results},
	eprint = {2201.03534},
	note = {preprint. \url{https://arxiv.org/abs/2201.03534}},
	year = {2022}}

@article {PP95,
	author = {Pillay, Anand and Poizat, Bruno},
	title = {Corps et chirurgie},
	journal = {J. Symbolic Logic},
	fjournal = {The Journal of Symbolic Logic},
	volume = {60},
	year = {1995},
	number = {2},
	pages = {528--533},
	doi = {10.2307/2275848},
	url = {https://doi.org/10.2307/2275848}}

@incollection {Wi75,
	author = {Winkler, Peter M.},
	title = {Model-completeness and {S}kolem expansions},
	booktitle = {Model theory and algebra ({A} memorial tribute to {A}braham {R}obinson)},
	series = {Lecture Notes in Math.},
	volume = {498},
	pages = {408--463},
	publisher = {Springer, Berlin-New York},
	year = {1975}}

@article {KMW26,
	author = {Kaplan, Elliot and Matthews, Angus and Walsberg, Erik},
	title = {Building trees in large fields},
	eprint = {2607.16942},
	note = {Preprint. \url{https://arxiv.org/abs/2607.16942}},
	year = {2026}}

@article {JTWY24,
	author = {Johnson, Will and Tran, Chieu-Minh and Walsberg, Erik and Ye, Jinhe},
	title = {The \'etale-open topology and the stable fields conjecture},
	journal = {J. Eur. Math. Soc. (JEMS)},
	fjournal = {Journal of the European Mathematical Society (JEMS)},
	volume = {26},
	year = {2024},
	number = {10},
	pages = {4033--4070},
	doi = {10.4171/jems/1345},
	url = {https://doi.org/10.4171/jems/1345}}
\end{document}